\documentclass[review,hidelinks,onefignum,onetabnum]{siamart250211}

\nolinenumbers

\usepackage{lipsum}
\usepackage{amsfonts}
\usepackage{graphicx}
\usepackage{epstopdf}
\usepackage{algorithm}                                                          
\usepackage{algpseudocode}
\usepackage{bm}
\usepackage{tabularx}
\usepackage[skip=0pt]{subcaption}
\ifpdf
  \DeclareGraphicsExtensions{.eps,.pdf,.png,.jpg}
\else
  \DeclareGraphicsExtensions{.eps}
\fi

\DeclareMathOperator{\ddiv}{div}

\newsiamremark{remark}{Remark}
\newsiamremark{hypothesis}{Hypothesis}
\crefname{hypothesis}{Hypothesis}{Hypotheses}
\newsiamthm{claim}{Claim}
\newtheorem{thm}{Theorem}

\headers{Convergence of Decoupled Schemes for Biot's Model}{X. Hu,  F. Gaspar,  C. Rodrigo}

\title{Convergence analysis for non-iterative sequential schemes for Biot's model
\thanks{Submitted to the editors in July 2024.
\funding{The work of  Hu is partially supported by the National Science Foundation (NSF) under grant DMS-2208267. . The work of Francisco J. Gaspar and Carmen Rodrigo is supported in part by the Spanish project PID2022-140108NB-I00 (MCIU/AEI/FEDER, UE), and by the DGA (Grupo de referencia APEDIF, ref. E24\_17R). 
}
}
}

\author{
Xiaozhe Hu\thanks{Tufts University, Medford, Massachusetts, USA (\email{xiaozhe.hu@tufts.edu}).}
\and Francisco J. Gaspar \thanks{IUMA and Applied Mathematics Department, University of Zaragoza, Zaragoza, Spain (\email{fjgaspar@unizar.es}, \email{carmenr@unizar.es}).}
\and Carmen Rodrigo \footnotemark[3]
}

\usepackage{amsopn}

\ifpdf
\hypersetup{
  pdftitle={TITLE},
  pdfauthor={F.J. Gaspar, X. Hu, C. Rodrigo}
}
\fi

\begin{document}

\maketitle

\begin{abstract}
An alternative to the fully implicit or monolithic methods used for the solution of the coupling of fluid flow and deformation in porous media is a sequential approach in which the fully coupled system is broken into subproblems (flow and mechanics problems) that are solved one after the other. This fully explicit coupling approach is a very simple scheme which allows much flexibility in the implementation and has a lower computational cost, making it quite attractive in practice since smaller linear systems need to be solved in order to obtain the solution for the whole coupled poroelastic system. Due to the appealing advantages of these methods, intensive research is currently being carried out in this direction, as in the present work. Although the application of this type of method is very common in practice, there exist only a few works devoted to their theoretical analysis. In this work, we consider the so-called explicit fixed-stress split scheme, which consists of solving the flow problem first with time-lagging the displacement term, followed by the solution of the mechanics problem. To the best of our knowledge, we provide the first convergence analysis of the explicit fixed-stress split scheme for Biot's equations. In particular, we prove that this algorithm is optimally convergent if the considered finite element discretization satisfies an inf-sup condition. In addition, with the aim of designing the simplest scheme for solving Biot's model, we also propose a similar decoupled algorithm for piecewise linear finite elements for both variables which arises from the novel stabilization recently proposed in \cite{Pe2025}, and is demonstrated to be optimally convergent.
\end{abstract}



\section{Introduction}\label{sec:intro}

The coupling of fluid flow and mechanical deformation within porous media is a relevant multi-physics problem for many real applications in different fields. Some examples of these important societal applications include geothermal energy extraction, CO2 storage, hydraulic fracturing, and cancer research, among others. Such a coupling was already modeled in the early one-dimensional work of Terzaghi \cite{terzaghi}, whereas Maurice Biot was the one who established its general three-dimensional mathematical formulation in several pioneering publications (see \cite{biot1},\cite{biot2}). This is the reason why this problem is known as Biot's model in the literature, as well as the poroelasticity problem. 

In this work, we consider the quasi-static Biot's model for soil consolidation, where the porous medium is assumed to be linearly elastic, homogeneous, isotropic, and saturated by a Newtonian fluid. Different formulations of this model can be considered depending on which variables one is interested in. The so-called two-field formulation, in which the main variables are the solid displacement and the fluid pressure, is widely used, and it is the one considered in this work. Another formulation which is widely used in practice is  the classical three-field formulation where the Darcy velocity is also included as a main variable. But, of course, there are other different interesting formulations, among which we can highlight  the three-field formulation in which the solid pressure is included as a third variable, or the recently introduced three-field formulation where  the total pressure, which is a weighted sum of fluid and solid pressures, is considered as an unknown (see \cite{Lee}). 

The classical two-field formulation of the model is given by
\begin{eqnarray}
& & - \mathbf{\nabla} \left(2\mu \varepsilon(\mathbf{u}) + \lambda \mathbf{\nabla} \cdot \mathbf{u} \right) + \alpha \mathbf{\nabla} p = \mathbf{f}, \label{biot1} \\
& &  \frac{1}{\beta} \partial_t p + \partial_t\left(\alpha \mathbf{\nabla} \cdot \mathbf{u} \right) - \nabla \cdot \left(K (\nabla p - \rho_f \mathbf{g}) \right) = g, \label{biot2}
\end{eqnarray}
on a space-time domain $\Omega \times (0,T]$, where $\Omega \subset \mathbb{R}^d$, $d \leq 3$, $T > 0$, and we denote the time derivative as $\partial_t$. In~\eqref{biot1}-\eqref{biot2}, $\mu$ and $\lambda$ are the Lam\'e parameters, $\mathbf{u}$ is the displacement vector, $p$ is the fluid pressure, $\varepsilon(\mathbf{u}) = \frac{1}{2} (\mathbf{\nabla} \mathbf{u} + \mathbf{\nabla}^T \mathbf{u})$ is the linearized strain tensor, $\alpha$ is the Biot cofficient, $\beta$ is the Biot modulus, $\rho_f$ is the fluid density, $\mathbf{g}$ is the gravity tensor,  $K $ is the hydraulic conductivity, and $\mathbf{f}$ and $g$ are the source terms for the mechanic and flow problems, respectively.  

Although both finite difference and finite volume methods have been applied successfully for the discretization of Biot’s model \cite{Gaspar2003, Nordbotten2016}, the finite element method is the most widely used numerical technique for solving such a multi-physics problem. Various finite element methods have been proposed in the literature and appropriately deal with the numerical difficulties that it exhibits. Stable finite-element schemes have been developed for the different formulations of Biot’s model. For the classical two-field formulation, Taylor–Hood elements, which satisfy an  inf-sup condition, have been studied \cite{MuradLoula92, MuradLoulaThome}, and appropriate stabilization techniques have also been proposed for unstable finite-element pairs, such as the combination of piecewise linear elements for both variables \cite{Pe2025, 2016RodrigoGasparHuZikatanov-a}. 


After discretization of the model, there exist three different approaches which are frequently applied to solve these problems  (see \cite{10.2118/79709-PA} to see a comparison of the three approaches emphasizing their advantages and drawbacks).  The first class of  algorithms includes the so-called fully implicit or monolithic methods in which all the unknowns are solved simultaneously on each time level. The main drawback of this type of scheme is its high computational cost due to the need to solve complex ill-conditioned linear systems, which requires the design of suitable preconditioners to accelerate the convergence of  Krylov subspace methods \cite{Adler2020, Berga2007, Bærland2017, Chen2020, Ferronato, Haga, Hong} or the design of appropriate smoothers within a multigrid framework \cite{Adler2023, Gaspar2017, peiyao_poro}. 

The second type of methods is the iterative coupling scheme in which, on each time step, either the flow or the mechanics part is solved first, followed by the solution of the other subproblem, repeating this process until a converged solution within a prescribed tolerance is obtained \cite{iterative_coupling}. These methods are quite attractive in practice, since smaller linear systems need to be solved for which one can apply well-known efficient linear solvers. In addition, their main advantage lies in the possibility of combining existing software for simulating fluid flow and geomechanics problems in order to obtain the solution for the whole coupled poroelastic system. The iterative coupling method most commonly used in practice is the so-called fixed-stress splitting method \cite{it_coup_fixed_stress}, which basically consists of solving the flow problem first by fixing the volumetric mean total stress, and then solving the mechanics part from the values obtained at the previous flow step. This method requires the choice of a certain stabilization parameter which has to be sufficiently large in order to ensure the convergence of the iteration \cite{BOTH2017101}. The fixed-stress split method has been rigorously shown to be unconditionally stable in the sense of a Von Neumann analysis \cite{it_coup_fixed_stress} and to be convergent \cite{BOTH2017101, Mikelic}. Depending on the problem, however, the convergence of the method may become slow \cite{Caste2015}. 

Finally, the third class of algorithms for solving Biot's model consists of the explicit or non-iterative sequential schemes in which the system is decoupled and no iterations are needed between the subproblems on each time step. The main advantages of these methods are their low computational cost compared with the previous algorithms and the simplicity of their implementation, but potentially at a cost of a less accurate numerical solution by using the same discretization parameters.  Due to the appealing advantages of these methods, intensive research is currently being carried out in this direction. In \cite{Lee2}, Lee has studied the stability and the convergence of two decoupled schemes for the three-field formulation, with the total pressure as an unknown, of the quasi-static multiple-network poroelasticity problem. In this way, instead of solving a very complex coupled problem, a linear elasticity equation and a system of parabolic equations have to be solved on each time step. Using the same three-field formulation but for the Biot's model, in \cite{osti_10471422} a decoupled scheme is proposed where the problem is split into a Stokes subproblem and a reaction-diffusion subproblem. This algorithm also has been extended to the case of the multiple-network poroelasticity model in \cite {ZHAO2025114214}.

Explicit coupling schemes have also been studied for the two-field formulation of Biot's model in several works. For example, in  \cite{ALMANI2019} the authors analyze a decoupled scheme based on first solving the mass equation \eqref{biot2} by time-lagging the displacement vector term and then by solving the mechanic problem. They  perform a stability analysis and derive the conditions needed for the physical parameters  to ensure stability for such a scheme. In addition, they observed that the explicit coupling schemes reduce the CPU run time significantly when compared to the iterative coupling schemes. For the same scheme, in \cite{CHAABANE2018} the authors derive a priori error estimates when stabilization terms are added to the mechanic equation. An interesting approach is presented  in \cite{Altmann} for different weakly coupled problems, such as the poroelaticity system. In that work, the convergence proofs are based on an interpretation of the explicit scheme as an implicit method applied to a constrained partial differential equation with a delay term. Finally, in \cite{KOLESOV20142185}, the authors obtain stability estimates of some splitting techniques by using Samarskii’s theory of stability for operator-difference schemes. 

The explicit coupling approach considered here is the so-called explicit fixed-stress split scheme, which is widely used in practice (see for example \cite{MINKOFF200337,10.2118/0401-0080-JPT, 10.2118/79709-PA} and the references therein). In such a decoupled method, the flow problem \eqref{biot2}  is solved first by time-lagging the displacement term, followed by solving the mechanics problem \eqref{biot1}. As in the iterative coupling fixed-stress split method, a stabilization term should be added in the flow equation for convergence reasons. To the best of our knowledge, the analysis presented in this work is the first convergence analysis of the explicit fixed-stress split scheme for Biot's equations. We prove that this algorithm is optimally convergent if the considered finite element discretization satisfies an inf-sup condition. In addition, with the aim of designing the simplest scheme for solving Biot's model, we also propose a similar algorithm for piecewise linear finite elements for both variables with the novel stabilization recently proposed in \cite{Pe2025}. We also prove in this work that such a scheme is  optimally convergent.


The structure of the paper is as follows. In \cref{sec:2}, we consider a discrete scheme satisfying an appropriate inf-sup condition, introduce the explicit fixed-stress split scheme, and present appropriate convergence estimates that prove the robust optimal convergence of the decoupled scheme for such a discretization. After that, \cref{sec:3} is devoted to dealing with one of the simplest discretizations for Biot's model: the P1-P1 scheme (piecewise linear finite element methods for both displacement and pressure). In this section we recall the stabilization proposed in~\cite{Pe2025} for the P1-P1 scheme, which naturally gives rise to a non-iterative coupling method similar to the previously introduced explicit fixed-stress split scheme, and demonstrate that such an algorithm is optimally convergent. \cref{sec:4} presents some numerical experiments that support the theoretical results previously introduced for both discrete schemes, and finally in \cref{sec:conclusions} some conclusions are drawn.

\section{Inf-sup Stable Discretization}
\label{sec:2}
In this section, we consider a stable finite element discretization of  the quasi-static Biot's model \eqref{biot1}-\eqref{biot2}. For simplicity, we assume homogeneous Dirichlet boundary conditions for both displacements and pressure on $\partial \Omega \times (0,T]$, and the initial conditions $\mathbf{u}(\mathbf{x},0) = \mathbf{u_0} (\mathbf{x})$, $p(\mathbf{x},0) = p_0 (\mathbf{x})$. According to the imposed boundary conditions, we consider the Sobolev spaces ${\mathbf V} = (H^1_0(\Omega))^d$ and $Q = H^1_0(\Omega)$, where $H^1_0(\Omega)$ denotes the Hilbert subspace of $L^2(\Omega)$ of functions with first weak derivatives in $L^2(\Omega)$ that are zero on the boundary of $\Omega$. We use $(\cdot,\cdot)$ to denote the inner product in $L^2(\Omega)$ and define the bilinear forms,
\begin{eqnarray*}\label{bilinear}
a(\bm{u},\bm{v}) = 2\mu \int_{\Omega}{ \varepsilon}(\bm{u}):{ \varepsilon}(\bm{v}) \, {\rm d} \Omega +
\lambda\int_{\Omega} \ddiv\bm{u}\ddiv\bm{v} \, {\rm d} \Omega, \qquad 
b(p,q) =  \int_{\Omega} K \nabla p \cdot \nabla q \, {\rm d} \Omega.
\end{eqnarray*}
The variational formulation for the two-field formulation of Biot's model is now: Find $({\mathbf u}(t), p(t))\in {\mathcal C}^1([0,T]; {\mathbf V}) \times {\mathcal C}^1([0,T];Q)$, such that,
\begin{eqnarray}
  && a(\bm{u},\bm{v}) -\alpha (p, \ddiv \bm{v})  
  = (\bm{f},\bm{v}),
     \quad \forall \  \bm{v} \in \bm V, \label{variational1}\\
  &&  \frac{1}{\beta} (\partial_t{p},q) + \alpha(\ddiv \partial_t{\bm{u}},q)  + b(p,q)  = (g,q),
   \; \forall \ q \in Q.\label{variational2}
\end{eqnarray}
We now describe the numerical approximation of this problem. Let $\mathcal{T}_h$ be a partition of $\Omega\subset \mathbb{R}^d$ consisting of triangles ($d=2)$ or tetrahedrons ($d=3$). For the spatial discretization of the Biot's model, we choose an inf-sup stable finite-element pair of spaces ${\bf V}_h \times Q_h$ to approximate the displacement and the pressure, respectively, i.e., we assume that the pair of finite element spaces satisfies the  inf-sup condition 
 \begin{align}\label{ine_inf-sup}
 	\sup_{\bm{w}_h \in  \bm{V}_h} \frac{(\nabla \cdot \bm{w}_h, q_h)}{ \| \bm{w}_h \|_A} \geq \eta \frac{1}{\sqrt{\lambda + 2\mu/d}} \| q_h \|, \quad \forall \, q_h \in Q_h,
  \end{align}
 where $\eta > 0$  is a constant that does not depend on the mesh size or the physical parameters. Taylor-Hood finite element \cite{MuradLoula92} and MINI-element \cite{2016RodrigoGasparHuZikatanov-a} are some examples of discretizations satisfying inequality \eqref{ine_inf-sup}. \\
 
Given any inf-sup stable finite element spaces ${\bf V}_h \times Q_h$, the semi-dsicretization of problem \eqref{variational1}-\eqref{variational2} is: for each $t \in (0,T]$, find $(\bm{u}_h(t), p_h(t)) \in  {\bf V}_h \times Q_h$ such that
 \begin{eqnarray}
  && a(\bm{u}_h(t),\bm{v}_h) -\alpha (p_h(t), \ddiv \bm{v}_h)  
  = (\bm{f},\bm{v}_h ),
     \quad \forall \  \bm{v}_h \in \bm V_h, \label{variational1_semi}\\
  &&  \frac{1}{\beta} (\partial_t p_h(t),q_h) + \alpha(\ddiv {\partial_t \bm{u}_h(t)},q_h)  + b(p_h(t),q_h)  = (g,q_h),
   \; \forall \ q_h \in Q_h.\label{variational2_semi}
\end{eqnarray}

\subsection{Numerical Scheme}
We now focus on the time discretization of problem \eqref{variational1_semi}-\eqref{variational2_semi}. We consider a uniform partition of the time interval $(0, T]$, $t_j=j\tau$, $j=0,\ldots, N$, with time-step $\tau=T/N$, and let $(\bm{u}_h^{j},p_h^{j})$ be the approximation of $(\bm{u}_h(t), p_h(t))$ at time level $t_j$.  We consider  the backward Euler method to discretize the first term in \eqref{variational2_semi} and a forward Euler method for the second term in \eqref{variational2_semi}. Notice that the resulting fully discrete problem is an explicit coupling approach  in which on each time level the flow problem is solved first followed for the mechanics problem, and therefore, there is no coupling between both problems anymore, as shown in the following algorithm:
\begin{algorithm}[H]
\caption{Explicit coupling algorithm}  \label{alg:explicit}                              
\begin{algorithmic}                                                         
\For{$j=1,2,\ldots,N-1$}
\State {\bf Step 1:} Given $({\bm u}_h^{j},{\bm u}_h^{j-1},p_h^{j}) \in {\bm V}_h \times {\bm V}_h \times Q_h$, find $p_h^{j+1} \in Q_h$ such that
\begin{align*}  
   \frac{1}{\beta} \left(\frac{p_h^{j+1}-p_h^{j}}{\tau},q_h\right)&+ b(p_h^{j+1},q_h) \\[-.1in]
  & = -\alpha \left(\ddiv \frac{\bm{u}_h^{j}-\bm{u}_h^{j-1}}{\tau},q_h\right) + (g_h^{j+1},q_h),
   \quad \forall \ q_h \in Q_h, \label{total_discrete_variational_split_pressure_2} 
\end{align*}
\State {\bf Step 2:} Given $p_h^{j+1} \in Q_h$, find ${\bm u}_h^{j+1} \in {\bm V}_h$ such that
\begin{equation*} \label{total_discrete_variational_split_displacement_2}
a(\bm{u}_h^{j+1},\bm{v}_h) =  \alpha( p_h^{j+1}, \ddiv \bm{v}_h) + (\bm{f}_h^{j+1},\bm{v}_h),
     \quad \forall \  \bm{v}_h \in \bm V_h. 
\end{equation*}
\EndFor
\end{algorithmic}
\end{algorithm}
Notice that to obtain $p_h^{j+1}$, solutions of ${\bm u}_h^{j}$, ${\bm u}_h^{j-1}$ and $p_h^{j}$ are required. In particular, we need ${\bm u}_h^{0}$, ${\bm u}_h^{1}$ and $p_h^{1 }$ to get $p_h^{2}$. Since we can not use the above scheme to obtain $p_h^{1}$, a fully implicit scheme is used to obtain the solutions at the first time level, i.e., ${\bm u}_h^{1}$ and $p_h^{1 }$, are obtained by solving 
 \begin{align}
  & a(\bm{u}_h^1,\bm{v}_h) -\alpha (p_h^1, \ddiv \bm{v}_h)  
  = (\bm{f}_h^1,\bm{v}_h ),
     \quad \forall \  \bm{v}_h \in \bm V_h, \label{eqn:u_1} \\
  &  \frac{1}{\beta} \left(\frac{p_h^1-p_h^0}{\tau} ,q_h\right) + \alpha\left(\ddiv \frac{\bm{u}_h^{1}-\bm{u}_h^{0}}{\tau} ,q_h\right)  + b(p_h^1,q_h)  = (g_h^1,q_h),
   \; \forall \ q_h \in Q_h.\label{eqn:p_1}
\end{align}

As shown in \cite{ALMANI2019}, Algorithm \eqref{alg:explicit} is conditionally stable. In particular, the condition $1/\beta > \alpha^2/\lambda$ is required in order to ensure stability. Moreover, it was shown that instabilities can occur if such a condition is violated. For this reason, different stabilization techniques have been proposed. Among them, the explicit fixed-stress split scheme is the most used in practice. Such a scheme, still decoupled,  is based on adding two artificial terms in the flow equation, one being discretized by a backward Euler scheme and the other by a forward Euler scheme, as it can be seen in \cref{alg:iterative_stable} .

Here, $L$ is a tuning parameter chosen sufficiently large to ensure the convergence of the scheme. Again, since the solution of the pressure depends on the solutions at the two previous time steps, scheme \eqref{eqn:u_1}-\eqref{eqn:p_1} can be used to obtain the displacement and the pressure at the first time level. In the next subsection, a convergence analysis of the explicit fixed-stress split scheme for Biot's equations is presented. To the best of our knowledge, this is the first convergence analysis of the explicit fixed-stress split scheme for Biot's model.

\begin{algorithm}[H]
	\caption{Explicit fixed-stress split algorithm}  \label{alg:iterative_stable}                              
	\begin{algorithmic}                                                         
		\For{$j=1,2,\ldots,N-1$}
		\State {\bf Step 1:} Given $({\bm u}_h^{j},{\bm u}_h^{j-1},p_h^{j},p_h^{j-1}) \in {\bm V}_h \times {\bm V}_h \times Q_h \times Q_h $, find $p_h^{j+1} \in Q_h$ such that:
		\begin{eqnarray}  
			&  & \quad \ \frac{1}{\beta} \left(\frac{p_h^{j+1}-p_h^{j}}{\tau},q_h\right) + L \left(\frac{p_h^{j+1}-p_h^{j}}{\tau},q_h\right)  + b(p_h^{j+1},q_h)  \label{pressure}  \\
			&&=  -\alpha \left(\ddiv \frac{\bm{u}_h^{j}-\bm{u}_h^{j-1}}{\tau},q_h\right)  
			+ L \left(\frac{p_h^{j}-p_h^{j-1}}{\tau},q_h\right)  + (g_h^{j+1},q_h),
			\quad \forall \ q_h \in Q_h, \nonumber
		\end{eqnarray}
		\State {\bf Step 2:} Given $p_h^{j+1} \in Q_h$, find ${\bm u}_h^{j+1} \in {\bm V}_h$ such that
		\begin{equation} \label{displacement}
			a(\bm{u}_h^{j+1},\bm{v}_h) =  \alpha( p_h^{j+1}, \ddiv \bm{v}_h) + (\bm{f}_h^{j+1},\bm{v}_h),
			\quad \forall \  \bm{v}_h \in \bm V_h. 
		\end{equation}
		\EndFor
	\end{algorithmic}
\end{algorithm} 

\subsection{Convergence Analysis} \label{sec:in-sup-stable-convergence-analysis}
For the theoretical study of the convergence of \cref{alg:iterative_stable}, we first define the elliptic projections as usual. For any $t$, define $\bar{\bm{u}}_h(t)$ and $\bar{p}_h(t)$ satisfying the following,
\begin{align}
	a(\bar{\bm{u}}_h, \bm{v}_h) - \alpha (\bar{p}_h, \nabla \cdot \bm{v}_h) &= a(\bm{u}, \bm{v}_h) - \alpha (p, \nabla \cdot \bm{v}_h), \quad \forall \, \bm{v}_h \in \bm{V}_h, \label{eqn:elliptic-proj-u} \\
	b(\bar{p}_h, q_h) &= b(p, q_h), \quad \forall \, q_h \in Q_h. \label{eqn:elliptic-proj-p}
\end{align}
Now, we can decompose the errors at $t=t_j$, $e_{\bm{u}}^{j}  = \bm{u}(t_j) - \bm{u}_h^j$ and  $e_{p}^{j} = p(t_j) - p_h^j$  in two parts as follows,
 \begin{align*}
 	e_{\bm{u}}^{j} &= \bm{u}(t_j) - \bm{u}_h^j = (\bm{u}(t_j) - \bar{\bm{u}}_h(t_j)) - (\bm{u}_h^j - \bar{\bm{u}}_h(t_j)) := \rho_{\bm{u}}(t_j) - \theta_{\bm{u}}^j, \\
 	e_{p}^{j} &= p(t_j) - p_h^j = (p(t_j) - \bar{p}_h(t_j)) - (p_h^j - \bar{p}_h(t_j)) := \rho_{p}(t_j) - \theta_{p}^j,
 \end{align*}
 where $\rho_{\bm{u}}(t_j)$ and $\rho_{p}(t_j)$ denote the difference between the exact solution and the elliptic projection, respectively,  and $\theta_{\bm{u}}^j$ and $\theta_{p}^j$ represent the difference between the approximate solution and the elliptic projection, respectively.  Let us assume that the inf-sup stable finite-element pair $\mathbf{V}_h$ and $Q_h$ provide $k$-th and $\ell$-th order approximation to the displacement and the pressure, respectively. Then for the error of the elliptic projections (see \cite{Murad1} for details), we have, for all $t$, 
 \begin{align}
 	\| \rho_{\bm{u}} \|_A &\leq c h^{\min(k, \ell+1)} \left( |\bm{u}|_{k+1} + |p|_{\ell+1} \right), \label{ine:err-rho-u-A} \\
 	\| \rho_p \|_B & \leq ch^{\ell} |p|_{\ell+1} , \quad 
 	\| \rho_p \| \leq ch^{\ell+1} |p|_{\ell+1} \label{ine:err-rho-p-B}
 \end{align}
Additionally, we have similar estimates for $\partial_t \rho_{\bm{u}}$ and $\partial_t \rho_p$, where on the right-hand side of the inequalities we have norms of $\partial_t \bm{u}$ and $\partial_t p$ instead of the norm of $\bm{u}$ and $p$, respectively. Here,  $\| \cdot \|_A$ and $\| \cdot \|_B$ are the norms associated with the bilinear forms $a(\cdot, \cdot)$ and $b(\cdot, \cdot)$, respectively.

 For the sake of simplicity, we introduce the following notation
 \begin{align*}
 	\delta \theta_{\bm{u}}^j &:= \theta_{\bm{u}}^j - \theta_{\bm{u}}^{j-1}, \quad d_t \theta_{\bm{u}}^j:= \frac{\delta \theta_{\bm{u}}^j}{\tau}, \quad  d_{tt} \theta_{\bm{u}}^j := \frac{d_t \theta_{\bm{u}}^{j+1} - d_t \theta_{\bm{u}}^j}{\tau}, \\
 	\delta \theta_{p}^j &:= \theta_{p}^j - \theta_{p}^{j-1}, \quad d_t \theta_{p}^j:= \frac{\delta \theta_{p}^j}{\tau}, \quad  d_{tt} \theta_{p}^j := \frac{d_t \theta_{p}^{j+1} - d_t \theta_{p}^j}{\tau}, \\
	R_p^{j} &: = \partial_t p(t_{j}) - d_t \bar{p}_h(t_{j}),  \quad \quad R_{\bm{u}}^{j} : = \partial_t \bm{u}(t_{j}) - d_t \bar{\bm{u}}_h(t_{j}). 
 \end{align*}
 We further assume that the initial condition satisfies 
 \begin{align}
 a(\bm{u}_h^0, \bm{v}_h) - \alpha (p_h^0, \nabla \cdot \bm{v}_h) &= (f_h^0, \bm{v}_h), \quad \forall \, \bm{v}_h \in \bm{V}_h, \label{eqn:ini-u} \\
 \alpha (\nabla \cdot \bm{u}_h^0, q_h) &= (g_h^0, q_h), \quad \forall \, q_h \in Q_h. \label{eqn:ini-p}
 \end{align}
Next, we prove the following lemmas, which will be used to establish the convergence of the numerical scheme.
 \begin{lemma} \label{est_time_0}
 The following estimate holds
 \begin{equation} \label{ine:A_norm_u>L2_norm_p}
	\| \delta \theta_{\bm{u}}^{j+1} \|^2_A \geq \eta^2\frac{\alpha^2}{\lambda + 2\mu/d} \| \delta \theta_p^{j+1} \|^2, \quad \text{for} \ j=0,1, \ldots, N-1. 
\end{equation}
 \end{lemma}
 \begin{proof}
  From the inf-sup condition \eqref{ine_inf-sup}, for any given $\delta \theta_p^{j+1}$, there exist a $\bm{w}_h \in \bm{V}_h$, such that
$
		(\nabla \cdot \bm{w}_h,  \delta \theta_p^{j+1}) \geq \eta \frac{1}{\sqrt{\lambda + 2\mu/d}} \| \delta \theta_p^{j+1} \| \| \bm{w}_h \|_A$, $\| \bm{w}_h \|_A = \| \delta \theta_p^{j+1} \|
		$.
Based on the definitions of $\theta_{\bm{u}}^j$ and $\theta_p^j$, from \eqref{variational1}, \eqref{displacement}, and \eqref{eqn:ini-u}, we have 
 \begin{align}
 &a(\theta_{\bm{u}}^{j}, \bm{v}_h) - \alpha(\theta_p^{j}, \nabla \cdot \bm{v}_h) = 0, \quad \forall \, \bm{v}_h \in \bm{V}_h,    \quad \text{for} \ j=0, 1, \ldots, N.    \label{eqn:theta_u_j+2} 
 \end{align}
Considering \eqref{eqn:theta_u_j+2} for the difference between $t=t_j$ and $t=t_{j+1}$, we arrive at
 \begin{equation} \label{eqn:delta_theta_u_j2}
 	a( \delta \theta_{\bm{u}}^{j+1}, \bm{v}_h) - \alpha(\delta \theta_p^{j+1}, \nabla \cdot \bm{v}_h) = 0, \quad \forall \, \bm{v}_h \in \bm{V}_h, \quad \text{for} \ j =0, 1, \ldots, N-1. 
 \end{equation}
Testing \eqref{eqn:delta_theta_u_j2} with $\bm{v}_h = \bm{w}_h$, we have
\begin{align*}
	 \quad  \frac{\alpha \eta}{\sqrt{\lambda + 2\mu/d}} \| \delta \theta_p^{j+1} \|  \| \bm{w}_h \|_A \leq \alpha (\delta \theta_p^{j+1}, \nabla \cdot \bm{w}_h) = a(\delta \theta_{\bm{u}}^{j+1}, \bm{w}_h) \leq \| \delta \theta_{\bm{u}}^{j+1} \|_A \| \bm{w}_h \|_A,
\end{align*}
which implies inequality \eqref{ine:A_norm_u>L2_norm_p}.
 \end{proof}
 \begin{lemma} \label{est_time_1}
If $L = \omega \frac{\alpha^2}{\lambda + 2\mu/d}$, $\omega \geq 1$, at the first time level, i.e., $j=1$, the following inequality holds
\begin{align}
	 \frac{L}{2} \| \delta \theta_p^1 \|^2 +  \frac{\tau}{2} \| \theta_p^1 \|_B^2 \leq \max \left\{ \frac{2\omega}{\eta^2},1 \right\} \left( \frac{\tau}{2} \| \theta_p^0 \|_B^2 + \frac{\tau^2}{4 \beta} \| R_p^1 \|^2 + \frac{\tau^2}{2\eta^2} \| R_{\bm{u}}^1 \|_A^2 \right). \label{ine:error_step_1}
\end{align}
 \end{lemma}
  \begin{proof}
  Using \eqref{eqn:delta_theta_u_j2},  \eqref{variational2}, \eqref{eqn:p_1}, and \eqref{eqn:ini-p}, we have
\begin{align*}
 &a(\delta \theta_{\bm{u}}^{1}, \bm{v}_h) - \alpha(\delta \theta_p^{1}, \nabla \cdot \bm{v}_h) = 0, \quad \forall \, \bm{v}_h \in \bm{V}_h  \\
&\frac{1}{\beta} (\delta \theta_{p}^{1}, q_h) + \alpha (\nabla \cdot \delta \theta_{\bm{u}}^{1}, q_h) + \tau \, b(\theta_p^{1}, q_h)  = \frac{\tau}{\beta} (R_p^{1}, q_h) + \tau \alpha(\nabla \cdot R_{\bm{u}}^{1}, q_h) , \quad \forall \, q_h \in Q_h, 
\end{align*}
Taking $\bm{v}_h = \delta \theta_{\bm{u}}^1$ and $q_h = \delta \theta_p^1$ and adding the above two equations, we arrive at
\begin{align*}
\| \delta \theta_{\bm{u}}^1 \|_A^2 + \frac{1}{\beta} \| \delta \theta_p^1 \|^2 + \tau b(\theta_p^1, \theta_p^1 - \theta_p^0) = \frac{\tau}{\beta} (R_p^1, \delta \theta_p^1) + \tau \alpha(\nabla \cdot R_{\bm{u}}^1, \delta \theta_p^1).
\end{align*}
Thus, using \eqref{ine:A_norm_u>L2_norm_p} ($j=0$), the Cauchy-Schwarz inequality, and the fact that $a(\bm{u}, \bm{u}) \geq (\lambda + 2\mu/d) \| \nabla \cdot \bm{u} \|^2$, we have
\begin{align*}
	& \quad \eta^2 \frac{\alpha^2}{\lambda + 2\mu/d} \| \delta \theta_p^1 \|^2 + \frac{1}{\beta} \| \delta \theta_p^1 \|^2 + \frac{\tau}{2} \| \theta_p^1 \|_B^2 \\
	& \leq \frac{\tau}{2} \| \theta_p^0 \|_B^2 + \frac{\tau^2}{4 \beta} \| R_p^1 \|^2 + \frac{1}{\beta} \| \delta \theta_p^1 \|^2 + \frac{\tau^2}{2\eta^2} \| R_{\bm{u}}^1 \|_A^2 + \frac{\eta^2}{2} \frac{\alpha^2}{\lambda + 2\mu/d} \| \delta \theta_p^1 \|^2. 
\end{align*}
Therefore, using that $L = \omega \frac{\alpha^2}{\lambda + 2\mu /d}$, $\omega \geq 1$, 
\begin{align*}
\frac{\eta^2}{2\omega} \frac{L}{2} \| \delta \theta_p^1 \|^2 +  \frac{\tau}{2} \| \theta_p^1 \|_B^2 \leq \frac{\tau}{2} \| \theta_p^0 \|_B^2 + \frac{\tau^2}{4 \beta} \| R_p^1 \|^2 + \frac{\tau^2}{2\eta^2} \| R_{\bm{u}}^1 \|_A^2.
\end{align*}
which leads to inequality \eqref{ine:error_step_1}.
  \end{proof}
  \begin{lemma} \label{est_time_2}
  The following estimates hold for $j=0, 1, \ldots, N$,
  \begin{align}
  	\tau^4 \| d_{tt} \bar{p}_h^j \|^2 \leq 4 \tau^3 \int_{t_{j-1}}^{t_{j+1}} \| \partial_{tt} p(s) \|^2 \, \mathrm{d}s + 2 \tau \int_{t_{j-1}}^{t_{j+1}} \| \partial_t \rho_p(s) \|^2 \, \mathrm{d}s, \label{ine:bound_dtt_p} \\
		\tau^4 \| d_{tt} \bar{\bm{u}}_h^j \|^2_A \leq 4 \tau^3 \int_{t_{j-1}}^{t_{j+1}} \| \partial_{tt} \bm{u}(s) \|_A^2 \, \mathrm{d}s + 2 \tau \int_{t_{j-1}}^{t_{j+1}} \| \partial_t \rho_{\bm{u}}(s) \|_A^2 \, \mathrm{d}s.  \label{ine:bound_dtt_u}
  \end{align}
  \end{lemma}
  \begin{proof}
We prove the estimate corresponding to $d_{tt}\bar{p}_h^j$, and a similar argument applies to the estimate for $d_{tt} \bar{\bm{u}}_h^j$.  Notice that
\begin{align*}
	\tau d_{tt} \bar{p}_h^j &= d_t \bar{p}_h^{j+1} - d_t  \bar{p}_h^{j} \\
	& = \left(  \partial_t p(t_{j+1}) - \frac{p(t_{j+1}) - p(t_j)}{\tau}  \right) +   \left(   \frac{\rho_p(t_{j+1}) - \rho_p(t_j)}{\tau} \right) - \partial_t p(t_{j+1}) \\
	& \quad  - \left(  \partial_t p(t_{j}) - \frac{p(t_{j}) - p(t_{j-1})}{\tau}  \right) -   \left(   \frac{\rho_p(t_{j}) - \rho_p(t_{j-1})}{\tau} \right) + \partial_t p(t_{j}) \\
	& = \frac{1}{\tau} \int_{t_j}^{t_{j+1}} (s - t_j) \partial_{tt} p(s) \, \mathrm{d}s  + \frac{1}{\tau} \int_{t_j}^{t_{j+1}} \partial_t \rho_p(s) \, \mathrm{d}s \\
	& \quad - \frac{1}{\tau} \int_{t_{j-1}}^{t_{j}} (s - t_{j-1}) \partial_{tt} p(s) \, \mathrm{d}s  - \frac{1}{\tau} \int_{t_{j-1}}^{t_{j}} \partial_t \rho_p(s) \, \mathrm{d}s - \int_{t_{j}}^{t_{j+1}} \partial_{tt} p(s) \, \mathrm{d}s.
\end{align*}
Therefore, we have
\begin{align*}
	\tau \| d_{tt} \bar{p}_h^j \| \leq 2 \int_{t_{j-1}}^{t_{j+1}} \| \partial_{tt} p(s) \| \, \mathrm{d}s + \frac{1}{\tau} \int_{t_{j-1}}^{t_{j+1}} \| \partial_t \rho_p(s) \| \, \mathrm{d}s
\end{align*}
and then inequality \eqref{ine:bound_dtt_p} yields.
 \end{proof}
 
We are now ready to present the error estimates.
  \begin{thm} \label{thm:convergence-alg2}
 Let $\bm{u}(t)$ and $p(t)$ be the solutions of \eqref{variational1} and \eqref{variational2}, and let $\bm{u}_h^n$ and $p_h^n$ be the solutions obtained by \cref{alg:iterative_stable}. If $L = \omega \frac{\alpha^2}{\lambda + 2\mu/d}$, $\omega \geq 1$, then the following error estimates hold,
 \begin{eqnarray}
&& \qquad \| \bm{u}(t_n) - \bm{u}_h^n \|_A + \| p(t_n) - p_h^n \|_B +  \| p(t_n) - p_h^n \|  \label{ine:alg2_error_u} \\
& &\quad  \leq C \left( \| e_{\bm{u}}^0 \|_A + \| e_p^0 \|_B  \right)  +  C h^s\left( |u(t_0)|_{k+1} + |p(t_0)|_{\ell+1}  + |u(t_n)|_{k+1} + |p(t_n)|_{\ell+1} \right)  \nonumber \\
&&  \qquad  + C \tau  \left[  \left( \int_{t_0}^{t_{n+1}} \| \partial_{tt} p(s) \|^2 \, \mathrm{d}s \right)^{\frac{1}{2}} +  \left( \int_{t_0}^{t_{n+1}} \| \partial_{tt} \bm{u}(s) \|_1^2 \, \mathrm{d}s \right)^{\frac{1}{2}}   \right] \nonumber \\
&& \qquad + Ch^s \left[  \left( \int_{t_0}^{t_{n+1}} | \partial_t p(s) |_{\ell+1}^2 \, \mathrm{d}s \right)^{\frac{1}{2}} + \left( \int_{t_0}^{t_{n+1}} |\partial_t \bm{u}(s)|^2_{k+1} + |\partial_t p(s)|_{\ell+1}^2 \, \mathrm{d}s \right)^{\frac{1}{2}} \right], \nonumber 
 \end{eqnarray}
 where $s = \min(k, \ell)$.
 \end{thm}
  \begin{proof}
   Based on \eqref{eqn:elliptic-proj-u} and \eqref{eqn:elliptic-proj-p}, from \eqref{variational2} and \eqref{pressure},  for $j\geq 1$, we have
 \begin{eqnarray}
 & \quad   \frac{1}{\beta} (\delta \theta_{p}^{j+1}, q_h) + \alpha (\nabla \cdot \delta \theta_{\bm{u}}^j, q_h) 
  + \tau \, b(\theta_p^{j+1}, q_h) + L (\delta \theta_p^{j+1} - \delta \theta_p^{j}, q_h) \label{eqn:theta_p_j+1} \\
 &   = \frac{\tau}{\beta} (R_p^{j+1}, q_h) + \tau \alpha(\nabla \cdot R_{\bm{u}}^{j+1}, q_h) - \tau^2 L(d_{tt} \bar{p}_h^j, q_h)  + \tau^2 \alpha (\nabla \cdot d_{tt} \bar{\bm{u}}_h^j, q_h), \ \forall \, q_h \in Q_h,  \nonumber
 \end{eqnarray}
 Now let $\bm{v}_h = \delta \theta_{\bm{u}}^{j}$ in \eqref{eqn:delta_theta_u_j2} and $q_h = \delta \theta_p^{j+1}$ in \eqref{eqn:theta_p_j+1}, and then adding the two equations, we have
 \begin{eqnarray}
 && \quad \qquad  a( \delta \theta_{\bm{u}}^{j+1}, \delta \theta_{\bm{u}}^{j}) + \frac{1}{\beta} \| \delta \theta_p^{j+1} \|^2 + L (\delta \theta_p^{j+1}, \delta \theta_p^{j+1} - \delta \theta_p^{j}) + \tau b(\theta_p^{j+1}, \theta_p^{j+1} - \theta_p^{j})  \label{def_of_I} \\
 && =   \frac{\tau}{\beta} (R_p^{j+1}, \delta \theta_p^{j+1}) + \tau \alpha(\nabla \cdot R_{\bm{u}}^{j+1}, \delta \theta_p^{j+1}) - \tau^2 L(d_{tt} \bar{p}_h^j, \delta \theta_p^{j+1}) + \tau^2 \alpha (\nabla \cdot d_{tt} \bar{\bm{u}}_h^j, \delta \theta_p^{j+1}) \nonumber \\
 && =:I.  \nonumber
 \end{eqnarray}
 Using the identity $a(a-b) = \frac{1}{2} (a^2-b^2 + (a-b)^2 )$ and a polarization identity, we have
 \begin{align*}
 & \quad 	\frac{1}{2} \| \delta \theta_{\bm{u}}^{j+1} \|_A^2 + \frac{1}{2} \| \delta \theta_{\bm{u}}^{j} \|_A^2 + \frac{1}{\beta} \| \delta \theta_p^{j+1} \|^2 + \frac{L}{2} \|\delta \theta_p^{j+1}\|^2 + \frac{\tau}{2} \|\theta_p^{j+1}\|_B^2 + \frac{\tau}{2} \| \delta \theta_p^{j+1} \|_B^2  \\
 & + \frac{L}{2} \| \delta \theta_p^{j+1} - \delta \theta_p^j \|^2 = \frac{L}{2} \|  \delta \theta_p^j \|^2 + \frac{\tau}{2} \| \theta_p^j \|_B^2  +  \frac{1}{2} \| \delta \theta_{\bm{u}}^{j+1}  - \delta \theta_{\bm{u}}^{j} \|_A^2+ I.
 \end{align*}
Considering \eqref{eqn:delta_theta_u_j2} for the difference of two consecutive time steps, $t= t_{j+1}$ and $t= t_{j}$,  and testing with $\bm{v}_h = \delta\theta_{\bm{u}}^{j+1} - \delta \theta_{\bm{u}}^j$, we obtain 
\begin{align*}
& \quad \| \delta\theta_{\bm{u}}^{j+1} - \delta \theta_{\bm{u}}^j \|_A^2 \\
&= \alpha (\delta \theta_p^{j+1} - \delta \theta_p^j, \nabla \cdot (\delta\theta_{\bm{u}}^{j+1} - \delta \theta_{\bm{u}}^j))  \leq \frac{\alpha}{\sqrt{\lambda+2\mu/d}} \| \delta \theta_p^{j+1} - \delta \theta_p^j \| \| \delta\theta_{\bm{u}}^{j+1} - \delta \theta_{\bm{u}}^j \|_A,
\end{align*}
where we used the Cauchy-Schwarz inequality and the fact that $a(\bm{u}, \bm{u}) \geq (\lambda + 2\mu/d) \| \nabla \cdot \bm{u} \|^2$.  Since $L = \omega \frac{\alpha^2}{\lambda + 2\mu /d}$, we have
\begin{align} \label{ine:u_normA_bound_by_p_norm}
	\| \delta\theta_{\bm{u}}^{j+1} - \delta \theta_{\bm{u}}^j \|_A \leq \frac{\alpha}{\sqrt{\lambda+2\mu/d}} \| \delta \theta_p^{j+1} - \delta \theta_p^j \| = \sqrt{\frac{L}{\omega}} \| \delta \theta_p^{j+1} - \delta \theta_p^j \| .
\end{align}
By \eqref{ine:u_normA_bound_by_p_norm} and \cref{est_time_0}, we arrive at
\begin{align*}
& \quad \frac{\eta^2}{2} \frac{\alpha^2}{\lambda + 2\mu/d} \| \delta \theta_p^{j+1} \|^2 + \frac{1}{2} \| \delta \theta_{\bm{u}}^j \|_A^2 + \frac{1}{\beta} \| \delta \theta_p^{j+1} \|^2  + \frac{L}{2} \|\delta \theta_p^{j+1}\|^2 + \frac{\tau}{2} \|\theta_p^{j+1}\|_B^2  \\
& + \frac{\tau}{2} \|  \delta \theta_p^{j+1}  \|_B^2  + \frac{L}{2} \| \delta \theta_p^{j+1} - \delta \theta_p^j \|^2 \leq \frac{L}{2} \|  \delta \theta_p^j \|^2 + \frac{\tau}{2} \| \theta_p^j \|_B^2 + \frac{L}{2\omega} \| \delta\theta_{p}^{j+1} - \delta \theta_{p}^{j} \|_A^2  + I. 
\end{align*}
Since we assume $\omega \geq 1$, we obtain
\begin{eqnarray}
	& &  \  \frac{\eta^2 L}{2\omega }\| \delta \theta_p^{j+1} \|^2 + \frac{1}{2} \| \delta \theta_{\bm{u}}^j \|_A^2 + \frac{1}{\beta} \| \delta \theta_p^{j+1} \|^2 + \frac{L}{2} \|\delta \theta_p^{j+1}\|^2 \label{ine:error_before_I}    \\
	& &+ \frac{\tau}{2} \|\theta_p^{j+1}\|_B^2 + \frac{\tau}{2} \|  \delta \theta_p^{j+1}  \|_B^2   \leq \frac{L}{2} \|  \delta \theta_p^j \|^2 + \frac{\tau}{2} \| \theta_p^j \|_B^2  + I.   \nonumber
\end{eqnarray}
To estimate $I$ defined in \eqref{def_of_I}, by applying the Young's inequality, 
we obtain
\begin{align*}
	\lvert I \rvert 
	& \leq \frac{\tau^2}{4\beta} \| R_{p}^{j+1} \|^2 + \frac{1}{\beta}  \| \delta \theta_p^{j+1} \|^2 + \frac{3\tau^2}{2\eta^2} \| R_{\bm{u}}^{j+1} \|_A^2 +   \frac{\eta^2}{6 \omega} L \|\delta \theta_p^{j+1} \|^2      \\
	& \quad + \frac{3 \tau^4 \omega}{ 2\eta^2} L \| d_{tt} \bar{p}_h^j \|^2 + \frac{\eta^2}{6 \omega} L \|\delta \theta_p^{j+1} \|^2  + \frac{3 \tau^4}{ 2\eta^2} \| d_{tt} \bar{\bm{u}}_h^j \|_A^2 + \frac{\eta^2}{6 \omega} L \|\delta \theta_p^{j+1} \|^2 \\
	& = \frac{\eta^2}{ 2\omega} L \|\delta \theta_p^{j+1} \|^2  + \frac{1}{\beta}  \| \delta \theta_p^{j+1} \|^2 + \frac{\tau^2}{4\beta} \| R_{p}^{j+1} \|^2 + \frac{3\tau^2}{2\eta^2} \| R_{\bm{u}}^{j+1} \|_A^2 \\
	& \quad + \frac{3 \tau^4 \omega}{2\eta^2} L \| d_{tt} \bar{p}_h^j \|^2  +  \frac{3 \tau^4}{ 2\eta^2} \| d_{tt} \bar{\bm{u}}_h^j \|_A^2. 
\end{align*}
Plug back in \eqref{ine:error_before_I} and dropping the term $\frac{\tau}{2} \| \delta \theta_p^{j+1} \|_B^2$, we obtain
 \begin{align*}
   & \quad \frac{1}{2} \| \delta \theta_{\bm{u}}^j \|_A^2 +  \frac{L}{2} \|\delta \theta_p^{j+1}\|^2 + \frac{\tau}{2} \|\theta_p^{j+1}\|_B^2 
     \leq  \frac{L}{2} \|  \delta \theta_p^j \|^2 + \frac{\tau}{2} \| \theta_p^j \|_B^2   + H^{j+1},
 \end{align*}
where
 $ \displaystyle 
 H^{j+1} =  \frac{\tau^2}{4\beta} \| R_{p}^{j+1} \|^2 + \frac{3\tau^2}{2\eta^2} \| R_{\bm{u}}^{j+1} \|_A^2 + \frac{3 \tau^4 \omega}{2\eta^2} L \| d_{tt} \bar{p}_h^j \|^2  +  \frac{3 \tau^4}{ 2\eta^2} \| d_{tt} \bar{\bm{u}}_h^j \|_A^2.
 $
  Summing the above inequality from $j=1$ to $j=n$, we have
 \begin{align*}
 \frac{1}{2} \sum_{j=1}^{n}  \| \delta \theta_{\bm{u}}^j \|_A^2 +  \frac{L}{2} \|\delta \theta_p^{n+1}\|^2 + \frac{\tau}{2} \|\theta_p^{n+1}\|_B^2 \leq \frac{L}{2} \|\delta \theta_p^{1}\|^2 + \frac{\tau}{2} \|\theta_p^{1}\|_B^2 + \sum_{j=1}^n H^{j+1}.
 \end{align*}
 Next, we estimate $H^{j+1}$.  For the first two terms involving $R_p^{j+1}$ and $R_{\bm{u}}^{j+1}$, following \cite[Lemma 8]{2016RodrigoGasparHuZikatanov-a}, we have
\begin{align}
	\tau^2 \| R_p^{j+1} \|^2 &\leq 2 \left( \tau^3 \int_{t_j}^{t_{j+1}} \| \partial_{tt} p(s) \|^2 \, \mathrm{d}s  + \tau \int_{t_j}^{t_{j+1}} \| \partial_t \rho_p(s) \|^2 \, \mathrm{d}s \right),  \label{ine:bound_R_p} \\
	\tau^2 \| R_{\bm{u}}^{j+1} \|_A^2 &\leq 2 \left( \tau^3 \int_{t_j}^{t_{j+1}} \| \partial_{tt} \bm{u}(s) \|^2_A \, \mathrm{d}s  +  \tau \int_{t_j}^{t_{j+1}} \| \partial_t \rho_{\bm{u}}(s) \|^2_A \, \mathrm{d}s \right). \label{ine:bound_R_u}
\end{align}
For the last two terms involving $d_{tt} \bar{p}_h^j$ and $d_{tt} \bar{\bm{u}}_h^j$, we use \cref{est_time_2}. Thus, 
\begin{align*}
	 & \quad  \frac{1}{2} \sum_{j=1}^{n}  \| \delta \theta_{\bm{u}}^j \|_A^2 + \frac{L}{2} \|\delta \theta_p^{n+1}\|^2 + \frac{\tau}{2} \|\theta_p^{n+1}\|_B^2 \leq \frac{L}{2} \|\delta \theta_p^{1}\|^2 + \frac{\tau}{2} \|\theta_p^{1}\|_B^2  \nonumber\\
	 & + C \left[ \tau^3 \left( \frac{1}{\beta} + \frac{\omega L}{\eta^2}  \right) \int_{t_0}^{t_{n+1}} \| \partial_{tt}p(s) \|^2 \, \mathrm{d}s + \tau \left( \frac{1}{\beta} + \frac{\omega L}{\eta^2} \right) \int_{t_0}^{t_{n+1}} \| \partial_t \rho_p(s)\|^2 \, \mathrm{d}s \right. \nonumber \\
	 & \left. \qquad \quad \ + \frac{\tau^3}{\eta^2} \int_{t_0}^{t_{n+1}} \| \partial_{tt} \bm{u}(s) \|_A^2 \, \mathrm{d}s +  \frac{\tau}{\eta^2} \int_{t_0}^{t_{n+1}} \| \partial \rho_{\bm{u}}(s) \|_A^2 \, \mathrm{d}s \right]. \nonumber
\end{align*}
By \cref{est_time_1}  and the fact that $\theta_p^0 = \rho_p(t_0) - e_p^0$,  we have
\begin{eqnarray}
&& \qquad   \sum_{j=1}^{n}  \| \delta \theta_{\bm{u}}^j \|_A^2 + L \|\delta \theta_p^{n+1}\|^2 + \tau \|\theta_p^{n+1}\|_B^2  \leq C\max\{ \frac{2\omega}{\eta^2},1 \}  \frac{\tau}{2} \left( \| e_p^0 \|_B^2 + \| \rho_p(t_0) \|_B^2 \right) \label{ine:bound-theta_u^A-theta_p^B}  \\
&& \quad  + C \max\{ \frac{2\omega}{\eta^2},1 \}  \left(\frac{\tau^3}{\beta}  \int_{t_0}^{t_1} \| \partial_{tt} p(s) \|^2 \, \mathrm{d}s + \frac{\tau}{\beta} \int_{t_0}^{t_1} \| \partial_t \rho_{p}(s) \|^2 \, \mathrm{d}s   \right. \nonumber \\
&& \qquad \qquad \qquad \qquad \quad  \left. +  \frac{\tau^3}{\eta^2} \int_{t_0}^{t_1} \| \partial_{tt}{\bm{u}}(s) \|_A^2 \, \mathrm{d}s + \frac{\tau}{\eta^2}  \int_{t_0}^{t_1} \| \partial_{t} \rho_{\bm{u}}(s)\|_A^2 \, \mathrm{d}s \right) \nonumber \\
&& \quad + C \left[ \tau^3 \left( \frac{1}{\beta} + \frac{\omega L}{\eta^2}  \right) \int_{t_0}^{t_{n+1}} \| \partial_{tt}p(s) \|^2 \, \mathrm{d}s + \tau \left( \frac{1}{\beta} + \frac{\omega L}{\eta^2} \right) \int_{t_0}^{t_{n+1}} \| \partial_t \rho_p(s)\|^2 \, \mathrm{d}s \right. \nonumber \\
&& \left. \qquad \quad \ + \frac{\tau^3}{\eta^2} \int_{t_0}^{t_{n+1}} \| \partial_{tt} \bm{u}(s) \|_A^2 \, \mathrm{d}s +  \frac{\tau}{\eta^2} \int_{t_0}^{t_{n+1}} \| \partial \rho_{\bm{u}}(s) \|_A^2 \, \mathrm{d}s \right]. \nonumber
\end{eqnarray}
Since $\theta_{\bm{u}}^n = \sum_{j=1}^n \delta \theta_{\bm{u}}^j + \theta_{\bm{u}}^0$, then
\begin{eqnarray} 
&& \ \quad 	\| \theta_{\bm{u}}^n \|_A  \leq  \sum_{j=1}^n \| \delta \theta_{\bm{u}}^j \|_A + \| \theta_{\bm{u}}^0 \|_A \leq \frac{\sqrt{t_n}}{\sqrt{\tau}}   \left( \sum_{j=1}^n \| \delta \theta_{\bm{u}}^j \|_A^2    \right)^{1/2} + \| \theta_{\bm{u}}^0 \|_A \label{ine:bound-theta_u^A} \\
	 && \ \leq \| e_{\bm{u}}^0 \|_A + \| \rho_{\bm{u}}(t_0) \|_A + C\max\{ \frac{\sqrt{2\omega}}{\eta},1 \}  \left( \| e_p^0 \|_B + \| \rho_p(t_0) \|_B \right)  \nonumber  \\
	 && \ \quad  + C \max\{ \frac{\sqrt{2\omega}}{\eta},1 \}  \left[\frac{\tau}{\sqrt{\beta}} \left( \int_{t_0}^{t_1} \| \partial_{tt} p(s) \|^2 \, \mathrm{d}s \right)^{\frac{1}{2}} + \frac{1}{\sqrt{\beta}} \left(\int_{t_0}^{t_1} \| \partial_t \rho_{p}(s) \|^2 \, \mathrm{d}s \right)^{\frac{1}{2}}   \right.   \nonumber\\ 
	 && \ \qquad \qquad \qquad \qquad \quad  \left. +  \frac{\tau}{\eta} \left( \int_{t_0}^{t_1} \| \partial_{tt}{\bm{u}}(s) \|_A^2 \, \mathrm{d}s \right)^{\frac{1}{2}} + \frac{1}{\eta}  \left( \int_{t_0}^{t_1} \| \partial_{t} \rho_{\bm{u}}(s)\|_A^2 \, \mathrm{d}s \right)^{\frac{1}{2}} \right]  \nonumber\\ 
	& & \ \quad + C \left[ \tau \left( \frac{1}{\beta} + \frac{\omega L}{\eta^2}  \right)^{\frac{1}{2}} \left( \int_{t_0}^{t_{n+1}} \| \partial_{tt}p(s) \|^2 \, \mathrm{d}s \right)^{\frac{1}{2}} \right.  \nonumber \\
	 && \ \qquad \quad \ \left. +  \left( \frac{1}{\beta} + \frac{\omega L}{\eta^2} \right)^{\frac{1}{2}} \left( \int_{t_0}^{t_{n+1}} \| \partial_t \rho_p(s)\|^2 \, \mathrm{d}s \right)^{\frac{1}{2}} \right.  \nonumber \\
	 && \ \left. \qquad \quad \ + \frac{\tau}{\eta} \left( \int_{t_0}^{t_{n+1}} \| \partial_{tt} \bm{u}(s) \|_A^2 \, \mathrm{d}s \right)^{\frac{1}{2}} +  \frac{1}{\eta} \left( \int_{t_0}^{t_{n+1}} \| \partial_t \rho_{\bm{u}}(s) \|_A^2 \, \mathrm{d}s \right)^{\frac{1}{2}} \right],    \nonumber
\end{eqnarray}
from which \eqref{ine:alg2_error_u} follows, via the triangle inequality, the error estimates of the projections (\eqref{ine:err-rho-u-A} and \eqref{ine:err-rho-p-B}), the Poincar\'{e} inequality that $\|\theta_p^{n}\| \leq C_p \| \theta_p^{n} \|_B$, and the fact that \eqref{ine:bound-theta_u^A-theta_p^B} holds true when we replace $n$ by $n-1$.
 \end{proof}
 
 \begin{remark}
 	We emphasize that the intermediate error estimate \eqref{ine:bound-theta_u^A}  is para\-me\-ter-robust.  If the projection estimates \eqref{ine:err-rho-u-A} and \eqref{ine:err-rho-p-B} were parameter-robust in the $\| \cdot \|_A$-norm, $\| \cdot \|_B$-norm, and $L^2$-norm, respectively, then the overall error estimate \eqref{ine:alg2_error_u} can be made parameter-robust.
 \end{remark}

\section{P1-P1 Discretization}\label{sec:3}
This section analyzes the convergence of an analogous explicit coupling scheme applied to a stabilized P1-P1 finite element discretization (piecewise linear for both displacements and pressure). Accordingly, we define the displacements and the pressure spaces as follows,
\begin{eqnarray*}
&&{\bf V}_h= \{ {\bm v}_h \in (H_0^1(\Omega))^d \, | \, {\bm v}_h |_T \in (P_1)^d, \; \forall T \in \mathcal{T}_h \}, \\
 &&Q_h = \{ p_h \in H_0^1(\Omega) \, | \, p_h |_T \in P_1, \forall T \in \mathcal{T}_h \}.
\end{eqnarray*}
This pair of finite-element spaces satisfies the following weak inf-sup condition
 \begin{align}\label{ine:inf-sup_with_epsilon}
 	\sup_{\bm{w}_h \in  \bm{V}_h} \frac{(\nabla \cdot \bm{w}_h, q_h)}{ \| \bm{w}_h \|_A} \geq \eta \frac{1}{\sqrt{\lambda + 2\mu/d}} \| q_h \| - \epsilon \frac{1}{\sqrt{\lambda+2\mu/d}} h \| \nabla q_h \|, \quad \forall \, q_h \in Q_h,
  \end{align}
 where $\eta > 0$ and $\epsilon >0$ are constants that do not depend on the mesh size or the physical parameters (see  \cite{2016RodrigoGasparHuZikatanov-a} for details). In order to stabilize this scheme, in \cite{Pe2025}, the authors proposed the following semidiscrete approximation of the problem \eqref{variational1}-\eqref{variational2}: for each $t \in (0,T]$, find $(\bm{u}_h(t), p_h(t)) \in  {\bf V}_h \times Q_h$ such that
 \begin{eqnarray}
  && a(\bm{u}_h(t),\bm{v}_h) -\alpha (p_h(t), \ddiv \bm{v}_h)  
  = (\bm{f},\bm{v}_h ),
     \quad \forall \  \bm{v}_h \in \bm V_h, \label{variational1_semi_2}\\
  &&  \frac{1}{\beta} (\partial_t{p}_h(t),q_h) + L (\partial_t{p}_h(t),q_h)_0 - L (\partial_t{p}_h(t),q_h) +  \alpha(\ddiv {\partial_t{\bm{u}}_h(t)},q_h)  \label{variational2_semi_2}  \\
  && \qquad \qquad + b(p_h(t),q_h)  =  (g,q_h),
   \; \forall \ q_h \in Q_h.\nonumber
\end{eqnarray}
where $(\cdot,\cdot)_0$  is an approximation of the $L_2$ inner product defined by mass lumping.

\subsection{Numerical Scheme}
For the time discretization of problem \eqref{variational1_semi_2}-\eqref{variational2_semi_2}, we again consider a uniform partition of the time interval $(0, T]$, $t_j=j\tau$, $j=0,\ldots, N$, with time-step $\tau=T/N$, and let $(\bm{u}_h^{j},p_h^{j})$ be the approximation of $(\bm{u}_h(t), p_h(t))$ at time level $t_j$.  We now consider  the backward Euler method to discretize the first two terms in \eqref{variational2_semi_2} and the forward Euler method for the third and fourth terms in \eqref{variational2_semi_2}. Notice that, again, the resulting fully discrete problem leads to an explicit coupling approach in which, at each time level, the flow problem is solved first, followed by the mechanics problem. Therefore, there is no coupling between the two problems, as shown in the following algorithm:
\begin{algorithm}[H]
\caption{Explicit coupling algorithm for stabilized P1-P1}  \label{alg:iterative_stable_2}                              
\begin{algorithmic}                                                         
\For{$j=1,2,\ldots,N-1$}
\State {\bf Step 1:} Given $({\bm u}_h^{j},{\bm u}_h^{j-1},p_h^{j},p_h^{j-1}) \in {\bm V}_h \times {\bm V}_h \times Q_h \times Q_h $, find $p_h^{j+1} \in Q_h$
\begin{eqnarray}  
  && \label{pressure_mass_lumping}  \qquad  \quad \   \frac{1}{\beta} \left(\frac{p_h^{j+1}-p_h^{j}}{\tau},q_h\right) + L \left(\frac{p_h^{j+1}-p_h^{j}}{\tau},q_h\right)_0  + b(p_h^{j+1},q_h) \\
  &&  \qquad  =  -\alpha \left(\ddiv \frac{\bm{u}_h^{j}-\bm{u}_h^{j-1}}{\tau},q_h\right)  + L \left(\frac{p_h^{j}-p_h^{j-1}}{\tau},q_h\right)  + (g_h^{j+1},q_h),
   \quad \forall \ q_h \in Q_h,  \nonumber 
\end{eqnarray}
\State {\bf Step 2:} Given $p_h^{j+1} \in Q_h$, find ${\bm u}_h^{j+1} \in {\bm V}_h$ such that
\begin{equation} \label{displacement_mass_lumping}
a(\bm{u}_h^{j+1},\bm{v}_h) =  \alpha( p_h^{j+1}, \ddiv \bm{v}_h) + (\bm{f}_h^{j+1},\bm{v}_h),
     \quad \forall \  \bm{v}_h \in \bm V_h. 
\end{equation}
\EndFor
\end{algorithmic}
\end{algorithm} 
Since we cannot use the above scheme to compute  $p_h^{1}$, as before, the following fully implicit scheme is used to obtain the solutions ${\bm u}_h^{1}$ and $p_h^{1 }$ at the first time level, 
 \begin{eqnarray}
  && a(\bm{u}_h^1,\bm{v}_h) -\alpha (p_h^1, \ddiv \bm{v}_h)  
  = (\bm{f}_h^1,\bm{v}_h ),
     \quad \forall \  \bm{v}_h \in \bm V_h, \label{eqn:u_2} \\
  &&  \frac{1}{\beta} \left(\frac{p_h^1-p_h^0}{\tau} ,q_h\right) + L \left(\frac{p_h^{1}-p_h^{0}}{\tau},q_h\right)_0  + \alpha\left(\ddiv \frac{\bm{u}_h^{1}-\bm{u}_h^{0}}{\tau} ,q_h\right)  + b(p_h^1,q_h) \label{eqn:p_2}\\
  & &  \quad \quad - L \left(\frac{p_h^{1}-p_h^{0}}{\tau},q_h\right)  = (g_h^1,q_h),
   \; \forall \ q_h \in Q_h. \nonumber
\end{eqnarray}
Next, we derive the error estimate of \cref{alg:iterative_stable_2}. Note that, for this algorithm, there is no need to add  stabilization terms, as was done in \cref{alg:iterative_stable}. 

\subsection{Convergence Analysis}
In this section we provide a convergence analysis for \cref{alg:iterative_stable_2}.  We use the same elliptic projections and the corresponding error estimates introduced in \cref{sec:in-sup-stable-convergence-analysis} with $k=\ell = 1$. Next, we prove two lemmas which are modifications of \cref{est_time_0} and \cref{est_time_1} due to the fact that, instead of the usual inf-sup condition \eqref{ine_inf-sup}, we only have the weak inf-sup condition \eqref{ine:inf-sup_with_epsilon} for the P1-P1 discretization. 
 \begin{lemma} \label{est_time_0_with_epsilon}
For any $\vartheta >0 $, the following estimate holds for $j=0,\ldots, N-1$,
 \begin{equation} \label{ine:A_norm_u>L2_norm_p_with_epsilon}
	\| \delta \theta_{\bm{u}}^{j+1} \|^2_A \geq \frac{\eta^2}{1+ 2 \vartheta}\frac{\alpha^2}{\lambda + 2\mu/d} \| \delta \theta_p^{j+1} \|^2  - \frac{\epsilon^2}{2 \vartheta} \frac{\alpha^2}{\lambda + 2\mu/d} h^2 \| \nabla \delta \theta_p^{j+1} \|^2.
\end{equation}
 \end{lemma}
 \begin{proof}
From the weak inf-sup condition \eqref{ine:inf-sup_with_epsilon}, for any given $\delta \theta_p^{j+1}$, there exist a $\bm{w}_h \in \bm{V}_h$, with $\| \bm{w}_h \|_A = \| \delta \theta_p^{j+1} \|$, such that
\begin{align*}
		(\nabla \cdot \bm{w}_h,  \delta \theta_p^{j+1}) \geq \left( \eta \frac{1}{\sqrt{\lambda + 2\mu/d}} \| \delta \theta_p^{j+1} \| -\epsilon \frac{1}{\sqrt{\lambda + 2\mu/d}} h \| \nabla \delta \theta_p^{j+1} \| \right) \| \bm{w}_h \|_A, \quad 
\end{align*}
Based on the definitions of the elliptic projections, \eqref{variational1} and \eqref{displacement_mass_lumping}, following the same argument as in \cref{est_time_0}, we again have
 \begin{equation} \label{eqn:delta_theta_u_j2_we}
 	a( \delta \theta_{\bm{u}}^{j+1}, \bm{v}_h) - \alpha(\delta \theta_p^{j+1}, \nabla \cdot \bm{v}_h) = 0, \quad \forall \, \bm{v}_h \in \bm{V}_h, \quad \text{for} \ j = 0, 1, \ldots, N-1.
 \end{equation}
Testing \eqref{eqn:delta_theta_u_j2_we} with $\bm{v}_h = \bm{w}_h$, we have
\begin{align*}
& \quad \alpha \left( \eta \frac{1}{\sqrt{\lambda + 2\mu/d}} \| \delta \theta_p^{j+1} \| -\epsilon \frac{1}{\sqrt{\lambda + 2\mu/d}} h \| \nabla \delta \theta_p^{j+1} \| \right) \| \bm{w}_h \|_A \leq \alpha (\delta \theta_p^{j+1}, \nabla \cdot \bm{w}_h) \\
	& = a(\delta \theta_{\bm{u}}^{j+1}, \bm{w}_h) \leq \| \delta \theta_{\bm{u}}^{j+1} \|_A \| \bm{w}_h \|_A,
\end{align*}
which implies 
\begin{equation*}
		\| \delta \theta_{\bm{u}}^{j+1} \|_A \geq \eta \frac{\alpha}{\sqrt{\lambda + 2\mu/d}} \| \delta \theta_p^{j+1} \| - \epsilon \frac{\alpha}{\sqrt{\lambda + 2\mu/d}} h \| \nabla \delta \theta_p^{j+1} \|.
\end{equation*}
Finally, using Young's inequality with constant $\vartheta > 0$,  \eqref{ine:A_norm_u>L2_norm_p_with_epsilon} is obtained.
 \end{proof}
 
 \begin{lemma} \label{est_time_1_with_epsilon}
	If $L = \omega \frac{\alpha^2}{\lambda + 2\mu/d}$, $\omega \geq 1$, at the first time level, i.e., $j=1$, the following inequality holds
	\begin{eqnarray}
	  &&\quad 	\frac{L}{2} \| \delta \theta_p^1 \|^2 +  \frac{\tau}{2} \| \theta_p^1 \|_B^2 \leq \max \left\{ \frac{2(\omega+2\epsilon^2 C_2)}{\eta^2},1 \right\} \left[ \frac{\tau}{2} \| \theta_p^0 \|_B^2 + \frac{\tau^2}{4 \beta} \| R_p^1 \|^2 \right. \label{ine:error_step_1_with_epsilon}  \\ 
		&& \qquad \qquad \qquad \qquad \qquad \ \left.+ \left( \frac{\omega  + 2\epsilon^2 C_2}{2 \eta^2}  \right) \frac{\tau^2}{\omega} \| R_{\bm{u}}^1 \|_A^2  + \frac{L}{2C_1} h^2 \| \nabla \delta \bar{p}_h^1 \|^2\right].  \nonumber 
	\end{eqnarray}
\end{lemma}
\begin{proof}
	Using \eqref{eqn:delta_theta_u_j2_we},  \eqref{variational2}, \eqref{eqn:p_2}, and \eqref{eqn:ini-p}, we have
	\begin{align*}
		&a(\delta \theta_{\bm{u}}^{1}, \bm{v}_h) - \alpha(\delta \theta_p^{1}, \nabla \cdot \bm{v}_h) = 0, \quad \forall \, \bm{v}_h \in \bm{V}_h  \\
		&\frac{1}{\beta} (\delta \theta_{p}^{1}, q_h) + \alpha (\nabla \cdot \delta \theta_{\bm{u}}^{1}, q_h) + \tau \, b(\theta_p^{1}, q_h) + L(\delta \theta_p^1, q_h)_Z  \\
		& \quad = \frac{\tau}{\beta} (R_p^{1}, q_h) + \tau \alpha(\nabla \cdot R_{\bm{u}}^{1}, q_h) - L(\delta \bar{p}_h^1, q_h)_Z , \quad \forall \, q_h \in Q_h, 
	\end{align*}
	where  $\| \cdot \|_Z$ is the norm associated to the bilinear form $(\cdot,\cdot)_0 - (\cdot,\cdot)$.  Taking $\bm{v}_h = \delta \theta_{\bm{u}}^1$ and $q_h = \delta \theta_p^1$ and adding the above two equations, we obtain
	\begin{align*}
		&\quad \| \delta \theta_{\bm{u}}^1 \|_A^2 + \frac{1}{\beta} \| \delta \theta_p^1 \|^2 + \tau b(\theta_p^1, \theta_p^1 - \theta_p^0) + L\| \delta \theta_p^1 \|_Z^2 \\
		& = \frac{\tau}{\beta} (R_p^1, \delta \theta_p^1) + \tau \alpha(\nabla \cdot R_{\bm{u}}^1, \delta \theta_p^1) - L(\delta \bar{p}_h^1, \delta \theta_p^1)_Z.
	\end{align*}
	Thus, using \cref{est_time_0_with_epsilon} ($j=0$), $L = \omega \frac{\alpha^2}{\lambda + 2\mu /d}$, $\omega \geq 1$, the Cauchy-Schwarz inequality, and the fact that $a(\bm{u}, \bm{u}) \geq (\lambda + 2\mu/d) \| \nabla \cdot \bm{u} \|^2$, we have
	\begin{align*}
		& \quad \frac{\eta^2}{\omega (1+2 \vartheta)} L \| \delta \theta_p^1 \|^2 + \frac{1}{\beta} \| \delta \theta_p^1 \|^2 + \frac{\tau}{2} \| \theta_p^1 \|_B^2 + L \| \delta \theta_p^1 \|_Z^2 \\
		& \leq \frac{\tau}{2} \| \theta_p^0 \|_B^2 + \frac{\epsilon^2L}{2\omega \vartheta} h^2 \| \nabla \delta \theta_p^1 \|^2 - L(\delta \bar{p}_h^1, \delta \theta_p^1)_Z  + \frac{\tau}{\beta} (R_p^1, \delta \theta_p^1) + \tau \alpha(\nabla \cdot R_{\bm{u}}^1, \delta \theta_p^1) \\
		& \leq \frac{\tau}{2} \| \theta_p^0 \|_B^2 + \left( \frac{\epsilon^2C_2}{2\omega \vartheta}  +\frac{1}{2}\right) L \| \delta \theta_p^1 \|_Z^2 + \frac{L}{2} \| \delta \bar{p}_h^1\|_Z^2 + \frac{\tau}{\beta} (R_p^1, \delta \theta_p^1) + \tau \alpha(\nabla \cdot R_{\bm{u}}^1, \delta \theta_p^1),
 	\end{align*}
 	where we used the norm equivalence $C_1 \| p \|_Z^2 \leq h^2 \| \nabla p \|^2 \leq C_2 \|  p \|_Z^2$, which can be found in \cite{Pe2025}.  Choose $\vartheta = \frac{\epsilon^2 C_2}{\omega}$ and use the Young's inequalty, we obtain that
 	\begin{align*}
 		&  \frac{\eta^2}{(\omega + 2 \epsilon^2 C_2)} L \| \delta \theta_p^1 \|^2 + \frac{1}{\beta} \| \delta \theta_p^1 \|^2 + \frac{\tau}{2} \| \theta_p^1 \|_B^2 
 		 \leq \frac{\tau}{2} \| \theta_p^0 \|_B^2  + \frac{L}{2 C_1} h^2 \| \nabla \delta \bar{p}_h^1\|^2 \\
 		 & \qquad + \frac{\tau^2}{4 \beta} \| R_p^1 \|^2 + \frac{1}{\beta} \| \delta \theta_p^1 \|^2 + \left(  \frac{\omega + 2 \epsilon^2 C_2}{2 \eta^2} \right) \frac{\tau^2}{\omega} \| R_{\bm{u}}^1 \|_A^2 + \frac{\eta^2}{2(\omega + 2 \epsilon^2C_2)} L \| \delta \theta_p^1 \|^2
 	\end{align*}
 	which simplifies to \eqref{ine:error_step_1_with_epsilon} and completes the proof.
\end{proof}

We now present the error estimates for \cref{alg:iterative_stable_2} in the following theorem.
 \begin{thm} \label{thm:convergence-alg3}
	Let $\bm{u}(t)$ and $p(t)$ be the solutions of \eqref{variational1} and \eqref{variational2}, and $\bm{u}_h^n$ and $p_h^n$ be the solutions obtained by  \cref{alg:iterative_stable_2}. If $L = \omega \frac{\alpha^2}{\lambda + 2\mu/d}$, $\omega \geq 1$, then the following error estimates hold,
	\begin{eqnarray}
		&& \quad \| \bm{u}(t_n) - \bm{u}_h^n \|_A + \| p(t_n) - p_h^n \|_B  +  \| p(t_n) - p_h^n \|   \label{ine:alg3_error_u} \\
		&& \leq C \left( \| e_{\bm{u}}^0 \|_A + \| e_p^0 \|_B  \right) +  Ch \left( |u(t_0)|_2 +  |p(t_0)|_2   + |u(t_n)|_2 +  |p(t_n)_2| \right)  \nonumber \\
		&& \quad  + C \tau  \left[  \left( \int_{t_0}^{t_{n+1}} \| \partial_{tt} p(s) \|^2 \, \mathrm{d}s \right)^{\frac{1}{2}} +  \left( \int_{t_0}^{t_{n+1}} \| \partial_{tt} \bm{u}(s) \|_1^2 \, \mathrm{d}s \right)^{\frac{1}{2}}   \right] \nonumber \\
		&& \quad + Ch \left[  \left( \int_{t_0}^{t_{n+1}} | \partial_t p(s) |_2^2 \, \mathrm{d}s \right)^{\frac{1}{2}} + \left( \int_{t_0}^{t_{n+1}} |\partial_t \bm{u}(s)|^2_2 + |\partial_t p(s)|_2^2 \, \mathrm{d}s \right)^{\frac{1}{2}} \right. \nonumber \\
		&& \left. \qquad \qquad + \left(  \tau \sum_{j=0}^{n} |\partial_t p(t_{j+1})|_1^2 \right)^{\frac{1}{2}} + \tau  \left( \int_{t_0}^{t_{n+1}} | \partial_{tt} p(s) |_1^2 \mathrm{d}s \right)^{\frac{1}{2}} \right] . \nonumber 
	\end{eqnarray}
\end{thm}

\begin{proof}
Based on \eqref{eqn:elliptic-proj-u} and \eqref{eqn:elliptic-proj-p}, from \eqref{variational2} and  \eqref{pressure_mass_lumping},  we obtain that
\begin{eqnarray}
		& & \qquad \quad \frac{1}{\beta} (\delta \theta_{p}^{j+1}, q_h) + \alpha (\nabla \cdot \delta \theta_{\bm{u}}^j, q_h) + \tau \, b(\theta_p^{j+1}, q_h) + L (\delta \theta_p^{j+1} - \delta \theta_p^{j}, q_h) \label{eqn:theta_p_j+1_ml}  \\ 
		& &  + L(\delta \theta_p^{j+1}, q_h)_Z 
		= \frac{\tau}{\beta} (R_p^{j+1}, q_h) + \tau \alpha(\nabla \cdot R_{\bm{u}}^{j+1}, q_h) - \tau^2 L(d_{tt} \bar{p}_h^j, q_h) - L(\delta \bar{p}_h^{j+1}, q_h)_Z  \nonumber \\
		& &  \qquad \qquad \qquad  \qquad + \tau^2 \alpha (\nabla \cdot d_{tt} \bar{\bm{u}}_h^j, q_h), \ \forall \, q_h \in Q_h.  \nonumber 
\end{eqnarray}
Let $\bm{v}_h = \delta \theta_{\bm{u}}^{j}$ in \eqref{eqn:delta_theta_u_j2_we} and $q_h = \delta \theta_p^{j+1}$ in \eqref{eqn:theta_p_j+1_ml}, adding the two equations, we have
\begin{eqnarray*}
	&& a( \delta \theta_{\bm{u}}^{j+1}, \delta \theta_{\bm{u}}^{j}) + \frac{1}{\beta} \| \delta \theta_p^{j+1} \|^2 + L (\delta \theta_p^{j+1}, \delta \theta_p^{j+1} - \delta \theta_p^{j})  \tau b(\theta_p^{j+1}, \theta_p^{j+1} - \theta_p^{j}) + L \| \delta \theta_p^{j+1}  \|_Z^2 \\
	&& =   \frac{\tau}{\beta} (R_p^{j+1}, \delta \theta_p^{j+1}) + \tau \alpha(\nabla \cdot R_{\bm{u}}^{j+1}, \delta \theta_p^{j+1}) - \tau^2 L(d_{tt} \bar{p}_h^j, \delta \theta_p^{j+1}) -L(\delta \bar{p}_h^{j+1}, \delta \theta_p^{j+1})_Z \\
	&& \quad+ \tau^2 \alpha (\nabla \cdot d_{tt} \bar{\bm{u}}_h^j, \delta \theta_p^{j+1}) = I - L(\delta \bar{p}_h^{j+1}, \delta \theta_p^{j+1})_Z := II,
\end{eqnarray*}
where $I$ was defined in \eqref{def_of_I}. Using the identity $a(a-b) = \frac{1}{2} (a^2-b^2 + (a-b)^2 )$ and a polarization identity, we have
\begin{eqnarray*}
	&& \quad 	\frac{1}{2} \| \delta \theta_{\bm{u}}^{j+1} \|_A^2 + \frac{1}{2} \| \delta \theta_{\bm{u}}^{j} \|_A^2 + \frac{1}{\beta} \| \delta \theta_p^{j+1} \|^2 + \frac{L}{2} \|\delta \theta_p^{j+1}\|^2 + \frac{\tau}{2} \|\theta_p^{j+1}\|_B^2 + \frac{L}{2} \| \delta \theta_p^{j+1} - \delta \theta_p^j \|^2 \\
	&& \quad + \frac{\tau}{2} \| \delta \theta_p^{j+1} \|_B^2 + L \| \delta \theta_p^{j+1} \|_Z^2 = \frac{L}{2} \|  \delta \theta_p^j \|^2 + \frac{\tau}{2} \| \theta_p^j \|_B^2  +  \frac{1}{2} \| \delta \theta_{\bm{u}}^{j+1}  - \delta \theta_{\bm{u}}^{j} \|_A^2+ II.
\end{eqnarray*}
Note that \eqref{ine:u_normA_bound_by_p_norm} still holds. 
By applying \cref{est_time_0_with_epsilon}
and using the assumption that $L = \omega \frac{\alpha^2}{\lambda + 2\mu /d}$, with $\omega \geq 1$, we obtain
\begin{align*}
	& \quad    \frac{\eta^2}{(2\omega)(1+2 \vartheta)} L \| \delta \theta_p^{j+1} \|^2 + \frac{1}{2} \| \delta \theta_{\bm{u}}^j \|_A^2 + \frac{1}{\beta} \| \delta \theta_p^{j+1} \|^2 + \frac{L}{2} \|\delta \theta_p^{j+1}\|^2 + \frac{\tau}{2} \|\theta_p^{j+1}\|_B^2 \\
	&  \quad + \frac{\tau}{2} \|  \delta \theta_p^{j+1}  \|_B^2 + L \| \delta \theta_p^{j+1} \|^2_Z \\
	& \leq \frac{L}{2} \|  \delta \theta_p^j \|^2 + \frac{\tau}{2} \| \theta_p^j \|_B^2   + \frac{\epsilon^2 C_2}{4 \omega \vartheta} L  \|  \delta \theta_p^{j+1} \|^2_Z  - L(\delta \bar{p}_h^{j+1}, \delta \theta_p^{j+1})_Z + I \\
	& \leq  \frac{L}{2} \|  \delta \theta_p^j \|^2 + \frac{\tau}{2} \| \theta_p^j \|_B^2   + \left(\frac{\epsilon^2 C_2}{4 \omega \vartheta} + \frac{1}{2}\right)L  \|  \delta \theta_p^{j+1} \|^2_Z  + \frac{L}{2} \| \delta \bar{p}_h^{j+1} \|_Z^2 + I \\
	&\leq  \frac{L}{2} \|  \delta \theta_p^j \|^2 + \frac{\tau}{2} \| \theta_p^j \|_B^2   + \left(\frac{\epsilon^2 C_2}{4 \omega \vartheta} + \frac{1}{2}\right)L  \|  \delta \theta_p^{j+1} \|^2_Z  + \frac{L}{2 C_1} h^2 \|  \nabla \delta \bar{p}_h^{j+1} \|^2 + I
\end{align*}
where we used the norm equivalence $C_1 \| p \|_Z^2 \leq h^2 \| \nabla p \|^2 \leq C_2 \|  p \|_Z^2$ which can be found in \cite{Pe2025}.  By choosing $\vartheta = \frac{\epsilon^2 C_2}{2 \omega}$ and  dropping $ \frac{\tau}{2} \|  \delta \theta_p^{j+1}  \|_B^2$,  we obtain
\begin{eqnarray}
&& \quad \quad \frac{\eta^2 }{2(\omega + \epsilon^2 C_2)} L \| \delta \theta_p^{j+1} \|^2 + \frac{1}{2} \| \delta \theta_{\bm{u}}^j \|_A^2 + \frac{1}{\beta} \| \delta \theta_p^{j+1} \|^2 + \frac{L}{2} \|\delta \theta_p^{j+1}\|^2 + \frac{\tau}{2} \|\theta_p^{j+1}\|_B^2 \label{ine:intermediate-error}\\
&&  \quad \leq   \frac{L}{2} \|  \delta \theta_p^j \|^2 + \frac{\tau}{2} \| \theta_p^j \|_B^2  + \frac{L}{2 C_1} h^2 \|  \nabla \delta \bar{p}_h^{j+1} \|^2 + I.  \nonumber 
\end{eqnarray}
Next, we estimate $I$.  Similarly to the proof of \cref{thm:convergence-alg2}, we mainly use the Young's inequality (with different constants in order to cancel the first term of \eqref{ine:intermediate-error}), 
\begin{align*}
	\lvert I \rvert 
	& \leq \frac{\tau^2}{4\beta} \| R_{p}^{j+1} \|^2   + \frac{1}{\beta}  \| \delta \theta_p^{j+1} \|^2  + \frac{3(\omega + \epsilon^2 C_2)}{2\eta^2} \frac{\tau^2}{\omega} \| R_{\bm{u}}^{j+1} \|_A^2 +   \frac{\eta^2}{6(\omega +  \epsilon^2 C_2)} L \|\delta \theta_p^{j+1} \|^2    \\
	& \quad  + \frac{3(\omega + \epsilon^2 C_2)}{2\eta^2} \tau^4 L \| d_{tt} \bar{p}_h^j \|^2+ \frac{\eta^2}{6( \omega +  \epsilon^2 C_2)} L \|\delta \theta_p^{j+1} \|^2 \\
	& \quad  + \frac{3(\omega + \epsilon^2 C_2)}{2\eta^2} \frac{\tau^4}{\omega} \| d_{tt} \bar{\bm{u}}_h^j \|_A^2  + \frac{\eta^2}{6(\omega +  \epsilon^2 C_2)}  L \|\delta \theta_p^{j+1} \|^2 \\
	& = \frac{\eta^2}{2( \omega +  \epsilon^2 C_2)}  L \|\delta \theta_p^{j+1} \|^2  + \frac{1}{\beta}  \| \delta \theta_p^{j+1} \|^2 + \frac{\tau^2}{4\beta} \| R_{p}^{j+1} \|^2 + \frac{3(\omega + \epsilon^2 C_2)}{2\eta^2}   \frac{\tau^2}{\omega} \| R_{\bm{u}}^{j+1} \|_A^2  \\
	& \quad + \frac{3(\omega + \epsilon^2 C_2)}{2\eta^2}   \tau^4L \| d_{tt} \bar{p}_h^j \|^2  +  \frac{3(\omega + \epsilon^2 C_2)}{2\eta^2}   \frac{\tau^4}{ \omega} \| d_{tt} \bar{\bm{u}}_h^j \|_A^2. 
\end{align*}
By plugging this back into \eqref{ine:intermediate-error}, we have
\begin{align*}
	& \quad \frac{1}{2} \| \delta \theta_{\bm{u}}^j \|_A^2 + \frac{L}{2} \|\delta \theta_p^{j+1}\|^2 + \frac{\tau}{2} \|\theta_p^{j+1}\|_B^2 \\
	& \leq  \frac{L}{2} \|  \delta \theta_p^j \|^2 + \frac{\tau}{2} \| \theta_p^j \|_B^2  + \frac{L}{2 C_1} h^2 \|  \nabla \delta \bar{p}_h^{j+1} \|^2 + \frac{\tau^2}{4\beta} \| R_{p}^{j+1} \|^2 + \frac{3(\omega + \epsilon^2 C_2)}{2\eta^2}   \frac{\tau^2}{\omega} \| R_{\bm{u}}^{j+1} \|_A^2 \\
	& \quad   + \frac{3(\omega + \epsilon^2 C_2)}{2\eta^2}   \tau^4L \| d_{tt} \bar{p}_h^j \|^2  +  \frac{3(\omega + \epsilon^2 C_2)}{2\eta^2}   \frac{\tau^4}{ \omega} \| d_{tt} \bar{\bm{u}}_h^j \|_A^2 \\
	& := \frac{L}{2} \|  \delta \theta_p^j \|^2 + \frac{\tau}{2} \| \theta_p^j \|_B^2 + X^{j+1}. 
\end{align*}
 Summing the above inequality from $j=1$ to $j=n$, we have 
 \begin{align*}
 	\sum_{j=1}^n\frac{1}{2} \| \delta \theta_{\bm{u}}^j \|_A^2 + \frac{L}{2} \|\delta \theta_p^{n+1}\|^2 + \frac{\tau}{2} \|\theta_p^{n+1}\|_B^2 \leq \frac{L}{2} \|  \delta \theta_p^1 \|^2 + \frac{\tau}{2} \| \theta_p^1 \|_B^2 + \sum_{j=1}^{n}X^{j+1}.
 \end{align*}
 Most of the terms in $X^{j+1}$ have already been estimated in the proof of  \cref{thm:convergence-alg2}.   To estimate $ \|  \nabla \delta \bar{p}_h^{j+1} \|^2$, note that $\nabla \delta  \bar{p}_j^{j+1} = \tau \left(  \partial_t \nabla p(t_{j+1}) - \nabla R_p^{j+1}  \right)$, then
\begin{align*}
\|  \nabla \delta \bar{p}_h^{j+1} \|^2 \leq 2 \tau^2 \left(   \| \partial_t \nabla p(t_{j+1}) \|^2  + \| \nabla R_p^{j+1} \|^2   \right) 
\end{align*}
Following the same approach for estimating $R_p^{j+1}$, we have
\begin{align*}
 \tau^2 \| \nabla R_p^{j+1} \|^2 \leq 2 \left( \tau^3 \int_{t_j}^{t_{j+1}} \| \partial_{tt} \nabla p(s) \|^2 \mathrm{d}s + \tau \int_{t_j}^{t_{j+1}} \| \partial_t \nabla \rho_p(s) \|^2 \mathrm{d}s  \right).
\end{align*}
Thus, 
\begin{eqnarray*}
 && \quad	\|  \nabla \delta \bar{p}_h^{j+1} \|^2 \\
	&& \leq 2 \tau^2  \| \partial_t \nabla p(t_{j+1}) \|^2 + 4 \left( \tau^3 \int_{t_j}^{t_{j+1}} \| \partial_{tt} \nabla p(s) \|^2 \mathrm{d}s + \tau \int_{t_j}^{t_{j+1}} \| \partial_t \nabla \rho_p(s) \|^2 \mathrm{d}s  \right).
\end{eqnarray*}
Combining \cref{est_time_2} with the bounds \eqref{ine:bound_R_p} and \eqref{ine:bound_R_u}, we obtain that
\begin{eqnarray}
	&& \qquad \qquad \frac{1}{2} \sum_{j=1}^{n}  \| \delta \theta_{\bm{u}}^j \|_A^2 + \frac{L}{2} \|\delta \theta_p^{n+1}\|^2 + \frac{\tau}{2} \|\theta_p^{n+1}\|_B^2  \label{ine:error_before_ini_ml} \\
	&&\leq \frac{L}{2} \|\delta \theta_p^{1}\|^2 + \frac{\tau}{2} \|\theta_p^{1}\|_B^2  + C \left[ \tau^3 \left( \frac{1}{\beta} + \frac{(\omega+\epsilon^2C_2) L}{\eta^2}  \right) \int_{t_0}^{t_{n+1}} \| \partial_{tt}p(s) \|^2 \, \mathrm{d}s \right. \nonumber\\
	&& \left. \quad + \tau \left( \frac{1}{\beta} + \frac{(\omega + \epsilon^2 C_2) L}{\eta^2} \right) \int_{t_0}^{t_{n+1}} \| \partial_t \rho_p(s)\|^2 \, \mathrm{d}s + \frac{(\omega + \epsilon^2 C_2)\tau^3}{\eta^2 \omega} \int_{t_0}^{t_{n+1}} \| \partial_{tt} \bm{u}(s) \|_A^2 \, \mathrm{d}s \right. \nonumber \\
	&& \left.  \quad +  \frac{(\omega + \epsilon^2 C_2)\tau}{\eta^2 \omega} \int_{t_0}^{t_{n+1}} \| \partial \rho_{\bm{u}}(s) \|_A^2 \, \mathrm{d}s \right]  + C \frac{L}{C_1} \left[  h^2   \tau^2  \sum_{j=1}^{n} \| \partial_t \nabla p(t_{j+1}) \|^2 \right. \nonumber \\
	&& \left. \quad +  h^2 \left( \tau^3   \int_{t_0}^{t_{n+1}} \| \partial_{tt} \nabla p(s) \|^2 \mathrm{d}s + \tau \int_{t_0}^{t_{n+1}} \| \partial_t \nabla \rho_p(s) \|^2 \mathrm{d}s  \right)  \right]. \nonumber
\end{eqnarray}
Using the bounds at the first time level in \cref{est_time_1_with_epsilon} and \eqref{ine:error_before_ini_ml}, 
and following the same procedure as in the proof of \cref{thm:convergence-alg2}, we obtain
the final error estimate \eqref{ine:alg3_error_u}, via the triangle inequality, the error estimates of the projections (\eqref{ine:err-rho-u-A} and \eqref{ine:err-rho-p-B}), the Poincar\'{e} inequality $\|\theta_p^{n}\| \leq C_p \| \theta_p^{n} \|_B$, and the fact that \eqref{ine:error_before_ini_ml} holds true when $n$ is replaced  by $n-1$.
\end{proof}

 \begin{remark}
	Again,  if the projection estimates \eqref{ine:err-rho-u-A} and \eqref{ine:err-rho-p-B} were parameter-robust in the $\| \cdot \|_A$-norm, $\| \cdot \|_B$-norm, and $L^2$-norm, respectively, then the overall error estimate \eqref{ine:alg3_error_u} can be made parameter-robust.
\end{remark}

\section{Numerical Experiments}\label{sec:4}
In this section, we present two numerical experiments to validate the theoretical results of the proposed decoupled algorithms and compare their performance with fully implicit schemes. The first example employs a manufactured smooth solution, while the second explores the Barry \& Mercer benchmark. To study the explicit fixed-stress method, we use a stable finite element discretization. For the pressure, we use piecewise linear finite elements, i.e., 
$
 Q_h = \{ p_h \in H_0^1(\Omega) \, | \, p_h |_T \in P_1, \; \forall T \in \mathcal{T}_h \}.
$
For the displacements, we enrich the piecewise linear functions with element bubble functions, i.e. ${\bf V}_h = {\bf V}_l \oplus  {\bf V}_b$, where ${\bf V}_l$ consists of the space of piecewise linear continuous vector valued functions on $\Omega$, i.e.,
$
{\bf V}_l = \{ {\bm v}_h \in (H_0^1(\Omega))^d \, | \, {\bm v}_h |_T \in (P_1)^d, \; \forall T \in \mathcal{T}_h \}.
$
and ${\bf V}_b$ is the space of element bubble functions. This pair ${\bf V}_h \times Q_h$, known as the MINI-element~\cite{mini}, satisfies the inf-sup condition \eqref{ine_inf-sup} (see \cite{2016RodrigoGasparHuZikatanov-a}). Throughout the experiments, we denote the error of the decoupled solution (via \cref{alg:iterative_stable} or \ref{alg:iterative_stable_2}) by $u-u_h^{dec}$, and the error of the corresponding fully implicit scheme by $u-u_h^{fully}$.

\subsection{Two-dimensional  Problem with a Manufactured Solution}
The experiment considers a unit square domain $\Omega = (0,1) \times (0,1)$, with a final time of $T = 1.0$. We take the source terms, the initial
and Dirichlet boundary conditions such that the exact solution of problem \eqref{biot1}-\eqref{biot2} is as follows
\begin{eqnarray}
u(x,y,t) &=&  \sin(\pi x t) \, \cos(\pi y t) \, xy(1-x)(1-y), \nonumber \\
v(x,y,t) &=&  \cos(\pi x t) \, \sin(\pi y t) \, xy(1-x)(1-y), \nonumber \\
p(x,y,t) &=& \cos(t+x-y) \, x y (1-x)(1-y) \nonumber.
\end{eqnarray}
Here we choose physical parameters as $\lambda =1, \mu = 2, \alpha=1, \beta=100$, and $K=1$. 

First, we focus on the explicit fixed-stress method, i.e., \cref{alg:iterative_stable}, where the stabilization parameter $L$ is chosen as $L=\alpha^2/(\lambda + \mu) = 1/3$. As noted, the MINI-element is used for spatial discretization. A right triangular grid of mesh size $h$ is used. We compute the errors in the $L^2$-norm between the exact solution and the numerical solution for the pressure, i.e., $\| p - p_h^{dec}\|$, and the corresponding errors when the problem is solved with a fully implicit method, i.e., $\| p - p_h^{fully}\|$, for comparison. These errors are shown in \cref{table_decoupled} at $t = 1$, where the convergence rates for the explicit fixed-stress scheme are also displayed. We observe that as we refine the time step $\tau$ and the mesh size $h$ appropriately, the errors decrease approximately by half, indicating first-order convergence, as expected from \eqref{ine:alg2_error_u} in \cref{thm:convergence-alg2} with $k=\ell=1$. Moreover, the errors obtained with the fully implicit method and the explicit coupling algorithm are both of a similar order of magnitude. Additionally, in this table, we also show the corresponding displacement errors $\| u - u_h^{dec}\|_A$ and $\| u - u_h^{fully}\|_A$. These displacement errors show convergence rates that match the predictions of \cref{thm:convergence-alg2} as well.
\begin{table}[htbp]
\begin{center}
\resizebox{\textwidth}{!}{%
	\begin{tabular}{cccccccc} 
		\hline
		$h$ & $\tau$ & $\| p - p_h^{fully}\|$ & $\| p - p_h^{dec}\|$ & rate & $\| u - u_h^{fully}\|_A$ & $\| u - u_h^{dec}\|_A$ & rate \\    
		\hline
		$1/40$ & $1/10$ & $1.027e-03$ & $2.937e-03$  &  & $1.576e-03$  & $1.847e-03$  & \\
		$1/80$ & $1/20$ & $5.343e-04$ & $1.593e-03$ & $0.88$ & $4.343e-04$ & $6.775e-04$ & $1.45$\\
		$1/160$ & $1/40$ & $2.716e-04$ & $8.186e-04$ & $0.96$ & $1.374e-04$& $2.963e-04$ & $1.19$ \\     
		$1/320$ & $1/80$ & $1.368e-04$ & $4.135e-04$ & $0.99$ & $5.401e-05$& $1.380e-04$ & $1.10$  \\
		\hline
	\end{tabular}%
}
\end{center}
\caption{Pressure ($L^2$-norm) and displacement (energy norm) errors for the explicit fixed-stress (\cref{alg:iterative_stable}) and fully implicit schemes, with convergence rates for the explicit case, for varying $\tau$ and $h$.}
\label{table_decoupled}
\end{table}
   
Next, we consider the decoupled \cref{alg:iterative_stable_2} for the stabilized P1-P1 finite element method. Following the recommendation given in \cite{Pe2025}, the parameter $L$ in the stabilization terms is chosen as $L=(3/2)(\alpha^2/(\lambda + \mu)) = 1/2$. In \cref{table_decoupled_2}, we show the errors in the $L^2$-norm for the pressure, i.e., $\| p - p_h^{dec}\|$, and in the energy norm for the displacements, i.e., $\| u - u_h^{dec}\|_A$, for different space and time discretization parameters. Additionally, for comparison, we display the corresponding errors obtained for the fully implicit version, i.e., $\| p - p_h^{fully}\|$ for the pressure and $\| u - u_h^{fully}\|_A$ for the displacements. We also report the convergence rates for the decoupled algorithm, observing first-order convergence as expected from \cref{thm:convergence-alg3}. 
\begin{table}[htbp]
\begin{center}
\resizebox{\textwidth}{!}{
		\begin{tabular}{cccccccc}  
			\hline
		          $h$ & $\tau$ & $\| p - p_h^{fully}\|$ &   $\| p - p_h^{dec}\|$ & rate & $\| u - u_h^{fully}\|_A$ & $\| u - u_h^{dec}\|_A$ & rate\\    
			\hline
			  $1/40$ & $1/10$ & $1.060e-03$ & $3.017e-03$  &  & $8.338e-04$  & $1.213e-03$ & \\
			  $1/80$ & $1/20$ & $5.423e-04$ &  $1.619e-03$ & $0.90$ &  $2.428e-04$ & $5.363e-04$ & $1.18$\\
			  $1/160$ & $1/40$ & $2.736e-04$ & $8.269e-04$ & $0.97$  & $9.628e-05$ &  $2.602e-04$ & $1.04$\\ 	
			  $1/320$ & $1/80$ &  $1.373e-04$ & $4.164e-04$  & $0.99$ & $4.413e-05$ & $1.307e-04$ & $0.99$\\
			\hline
		\end{tabular}
	}
\end{center}
\caption{Pressure ($L^2$-norm) and displacement (energy norm) errors for the explicit fixed-stress (\cref{alg:iterative_stable_2}) and fully implicit schemes, with convergence rates for the explicit case, for varying $\tau$ and $h$.}
\label{table_decoupled_2}
\end{table}

\subsection{Barry \& Mercer's Problem}
A well-known poroelastic benchmark on a finite two-dimensional domain is Barry \& Mercer's problem~\cite{Barry}. It models the behavior of a rectangular uniform porous material $[0,a]\times[0,b]$ with a pulsating point source, drained on all sides, and with zero tangential displacements assumed on the whole boundary. The point-source corresponds to a sine wave and is given by \(f(t) = 2\upsilon \,\delta_{(x_0,y_0)}\sin(\upsilon \,t)\), where $\upsilon = \displaystyle\frac{(\lambda + 2\mu)K}{a\,b}$ and $\delta_{(x_0,y_0)}$ is the Dirac delta at the point $(x_0,y_0)$. In \cref{BarryMercer}, the computational domain and the boundary conditions are shown.
\begin{figure}[!htb]
\begin{center}
\includegraphics*[width = 0.45\textwidth]{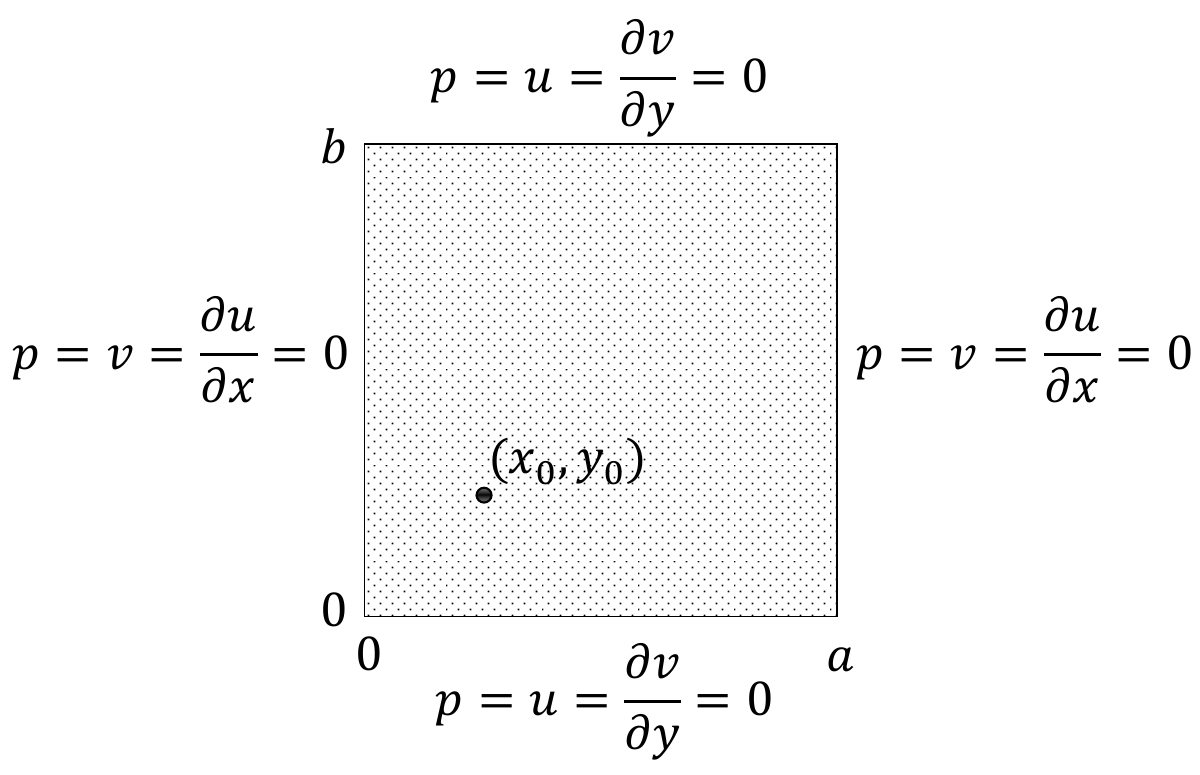}
\caption{The computational domain and boundary conditions for the Barry and Mercer's source problem.}
\label{BarryMercer}
\end{center}
\end{figure}
An analytical solution to this problem can be formulated in terms of a series expansion, see \cite{Barry} for details. We consider the square domain $(0,1)\times(0,1)$, i.e., $a=b=1$, and the physical  parameters are taken from \cite{osti_10471422}. In this way, Young's modulus and Poisson's ratio are chosen as  $E = 10^5$, $\nu = 0.1$, respectively, and the rest of parameters are fixed as $\alpha=1$, $\beta=0$ and $K = 10^{-2}$.  The source is positioned at the point $(1/4,1/4)$, and a right triangular grid of mesh size $h=2^{-6}$ is used for the simulations. The final time is set to $T=\pi/(2v)$, and the time step size is $\tau = \pi/(40v)$.  

First, we consider the explicit fixed-stress algorithm with stabilization parameter $L=\alpha^2/(\lambda + \mu)$ for the MINI-element discrete scheme. In \cref{Mini_BM}, we present the analytical solution and the numerical solutions obtained by using the decoupled \cref{alg:iterative_stable} and by using a fully implicit scheme, along the straight line $x = 0.25$. As we can observe, both numerical approximations exhibit good agreement with the analytical solution. Moreover, the solution obtained by the decoupled algorithm is very close to the one obtained by the coupled solution, demonstrating that the decoupled algorithm is a viable and effective choice for the simulations of Biot's model.
\begin{figure}[!htb]
\begin{center}
\begin{tabular}{ccc} 
\includegraphics*[width = 0.3\textwidth]{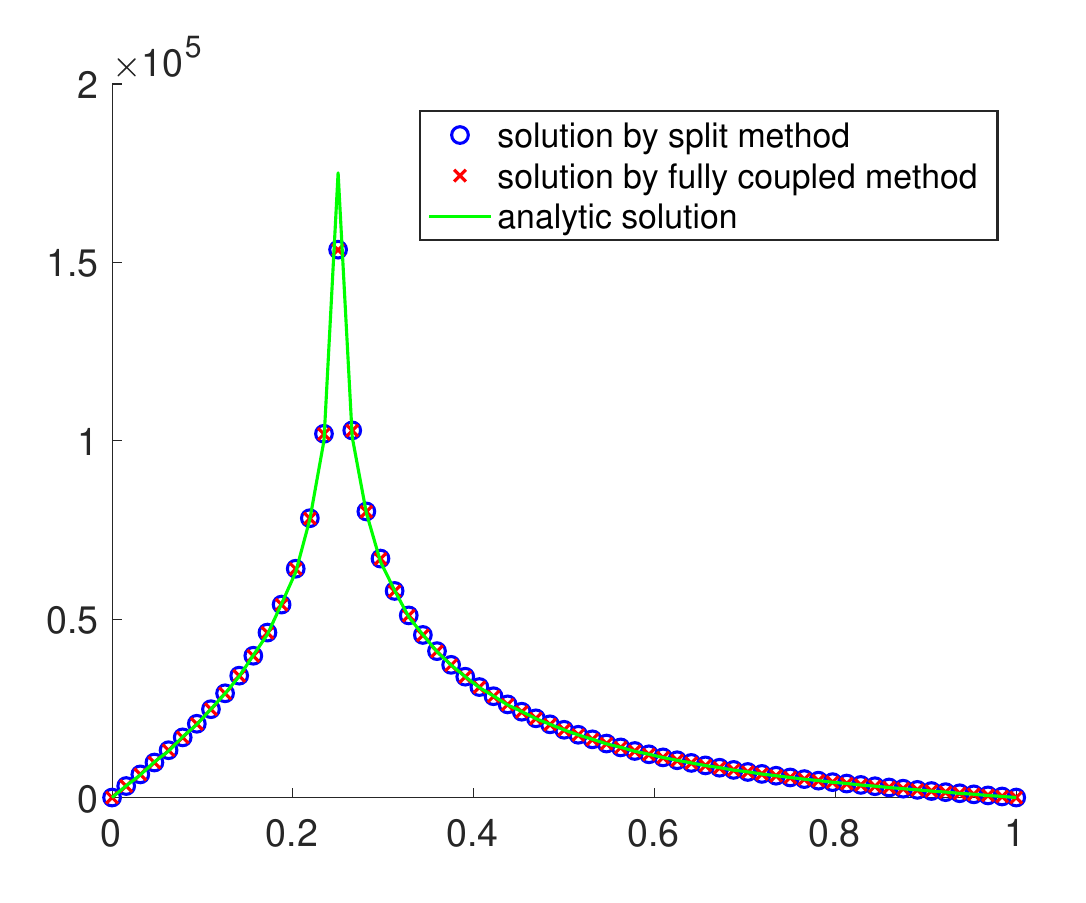} &
\includegraphics*[width = 0.3\textwidth]{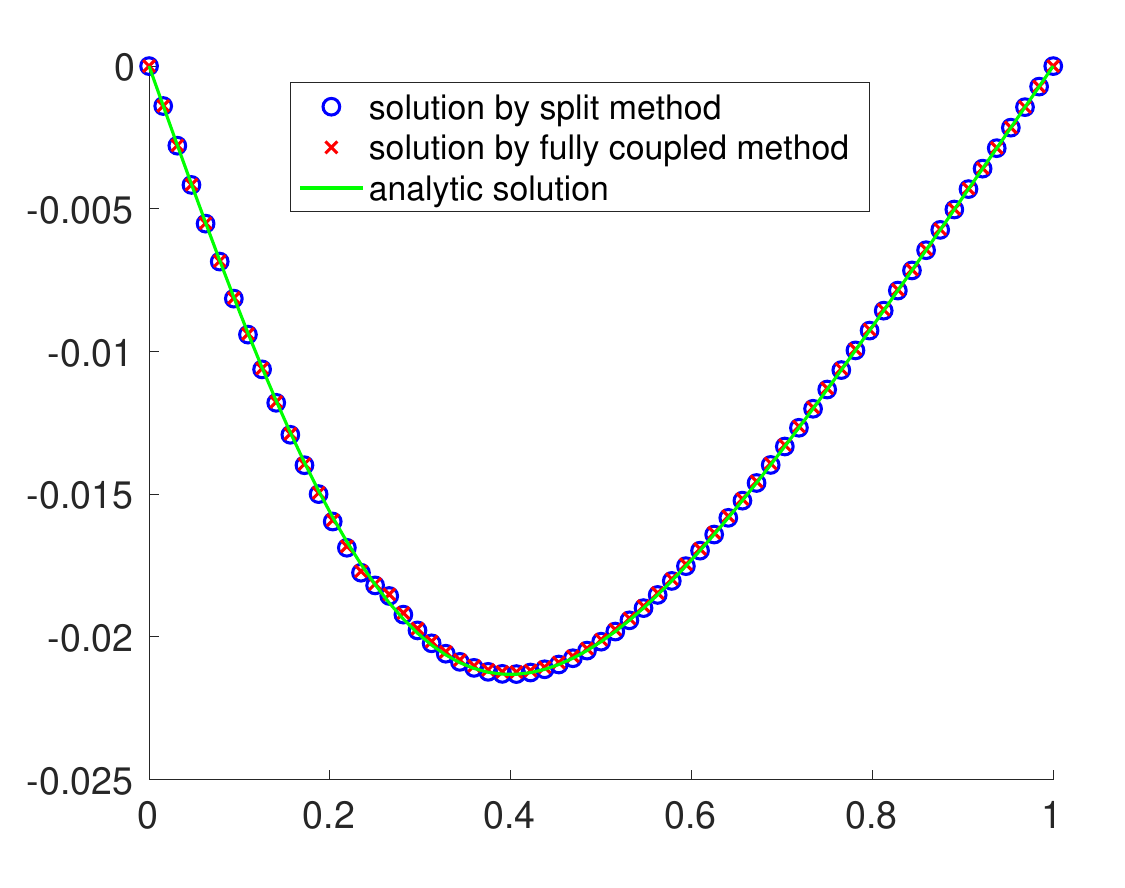}   & 
\includegraphics*[width = 0.3\textwidth]{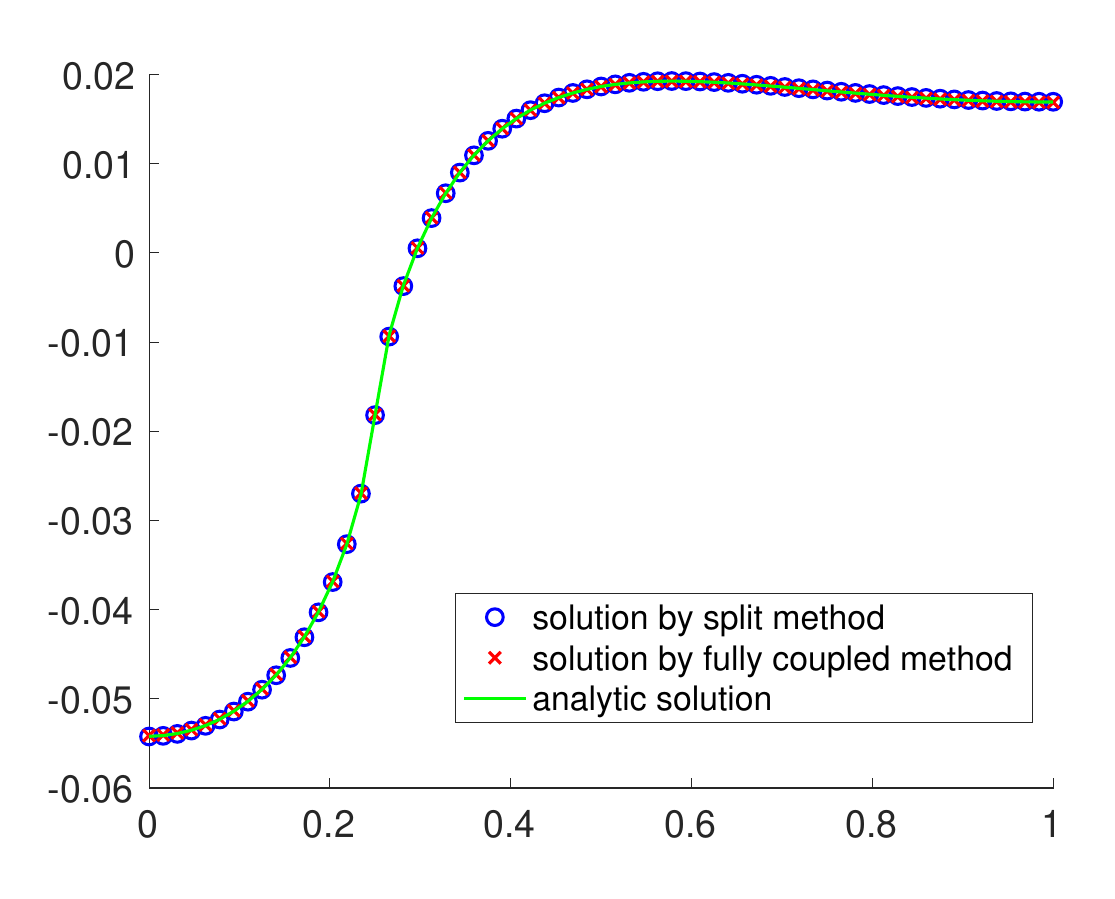}\\
(a) pressure $p$ &
(b) $x$-displacement $u$ &
(c) $y$-displacement $v$ 
\end{tabular}
\caption{Comparison of analytical, explicit split, and fully coupled solutions for the Barry \& Mercer problem using MINI-element: (a) pressure $p$, (b) x-displacement $u$, and (c) y-displacement $v$.}
\label{Mini_BM}
\end{center}
\end{figure}

Next, we consider the stabilized linear finite element discretization with stabilization parameter $L =(3/2)(\alpha^2/(\lambda + \mu))$. In \cref{P1_P1_BM}, we display the analytical solution together with the numerical solutions obtained using the decoupled \cref{alg:iterative_stable_2} and a fully implicit scheme, evaluated along the line $x = 0.25$. The obtained results are very satisfactory, as in the previous case. In fact, both numerical approximations show very good agreement with the analytical solution. In addition, the solution produced by the decoupled algorithm is very similar to that obtained with the fully coupled scheme, which supports the choice of the decoupled approach as a reliable and practical alternative for the numerical simulation of Biot’s model.

\begin{figure}[!htb]
\begin{center}
\begin{tabular}{ccc} 
\includegraphics*[width = 0.3\textwidth]{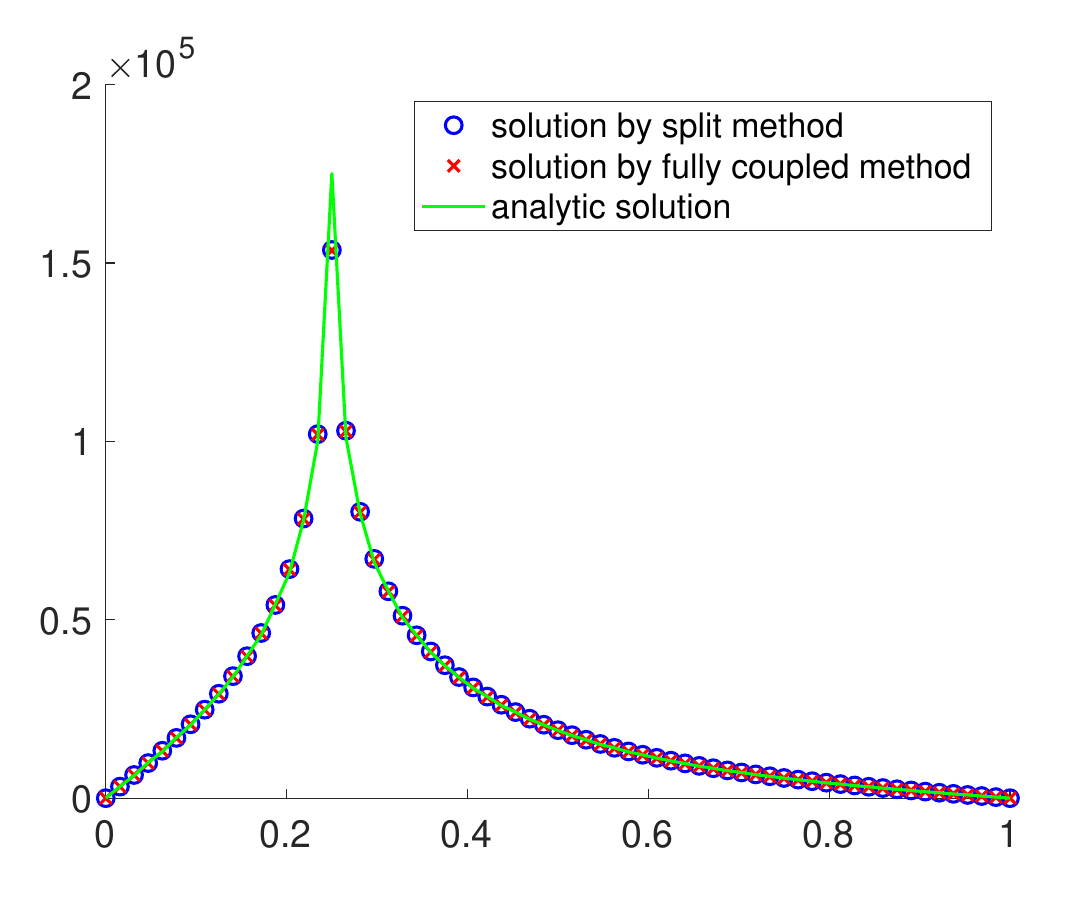} &
\includegraphics*[width = 0.3\textwidth]{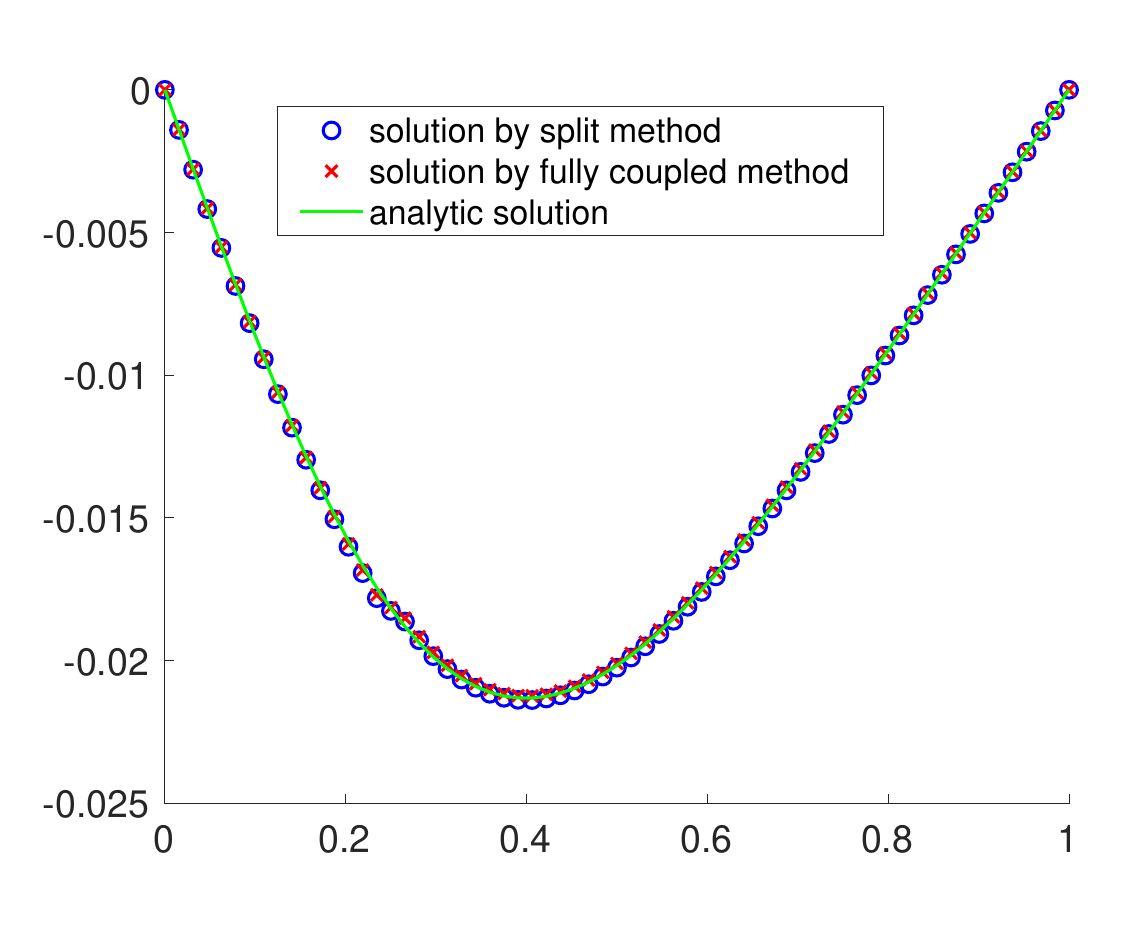}   & 
\includegraphics*[width = 0.3\textwidth]{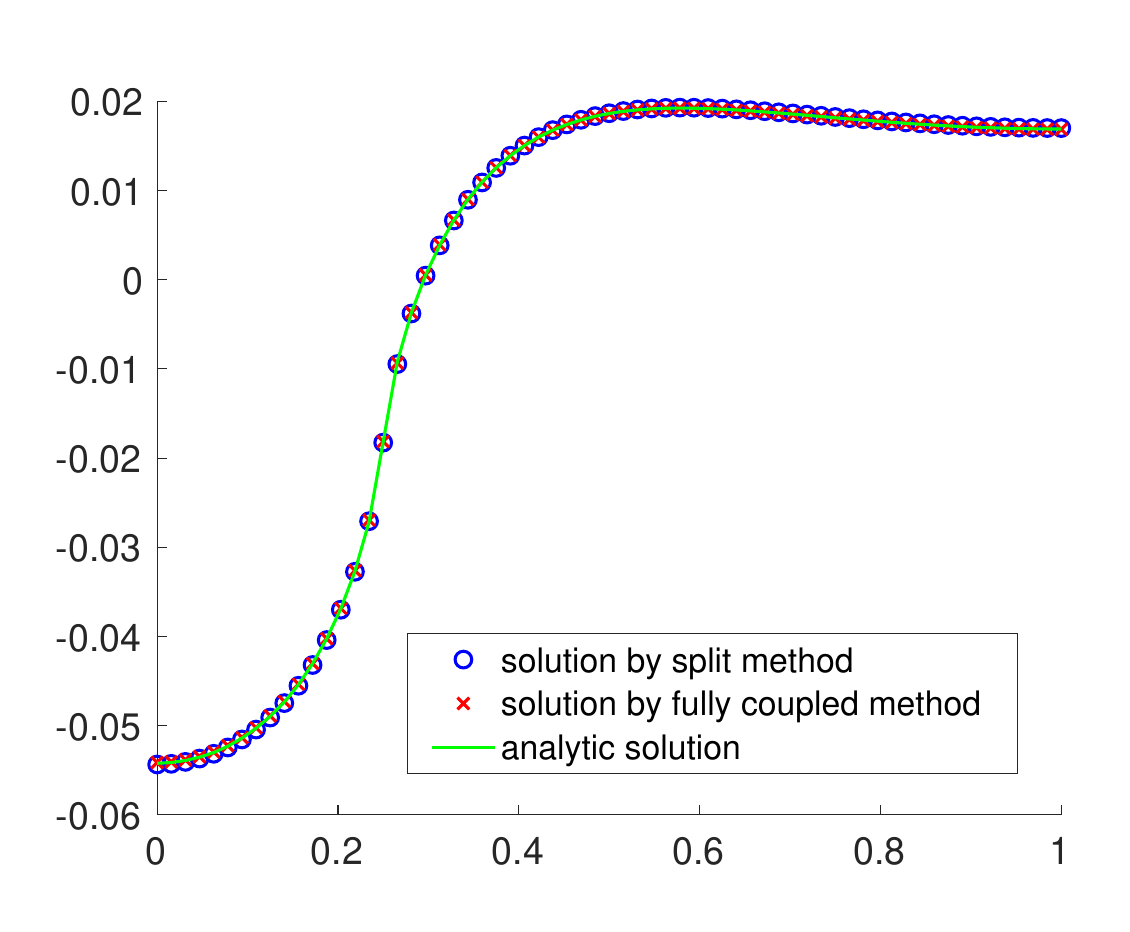} \\
(a) pressure $p$ &
(b) $x$-displacement $u$ &
(c) $y$-displacement $v$ 
\end{tabular}
\caption{Comparison of analytical, explicit split, and fully coupled solutions for the Barry \& Mercer problem using stabilized P1-P1 elements: (a) pressure $p$, (b) x-displacement $u$, and (c) y-displacement $v$.}
\label{P1_P1_BM}
\end{center}
\end{figure}

\section{Conclusions}\label{sec:conclusions}
In this work, we have investigated sequential alternatives to monolithic methods for solving the coupled flow–deformation problem in porous media. Specifically, we theoretically studied the explicit fixed-stress split scheme. Our results provide, to the best of our knowledge, the first convergence analysis of this split scheme for Biot’s equations. In particular, we have shown that the algorithm achieves optimal convergence, provided that the employed finite element discretization satisfies an inf–sup condition. Furthermore, with the aim of proposing a simple robust scheme for Biot’s model, we have introduced a similar decoupled algorithm based on a stabilization of piecewise linear finite elements for both variables, and we have demonstrated that this new scheme also exhibits optimal convergence. Overall, these findings confirm that explicit decoupling methods constitute a computationally efficient and theoretically sound alternative for the fully coupled poroelastic systems.

\bibliographystyle{siamplain}
\bibliography{mybibfile2}
\end{document}